\documentclass[11pt,a4paper]{article}

\usepackage{mathtools}
\usepackage{authblk} 
\usepackage{algorithmic,algorithm}

\usepackage{mathrsfs}
\usepackage{latexsym,bm}
\usepackage{amsmath,amsfonts,amssymb,amsthm}
\usepackage{extarrows}

\usepackage{graphicx,subfigure,epstopdf,float}
\usepackage{enumerate,cases,multirow}
\usepackage{makecell}
\usepackage{caption}

\usepackage{longtable,colortbl,arydshln,threeparttable}
\definecolor{mygray}{gray}{.9}

\usepackage{indentfirst}
\usepackage[top=25mm,bottom=20mm,left=25mm,right=20mm]{geometry}
\usepackage{cite}

\usepackage{listings}

\usepackage{makeidx}        
\usepackage{booktabs}
\usepackage[bookmarks,bookmarksnumbered,colorlinks,citecolor=red,linkcolor=red,hyperindex,linktocpage=true]{hyperref}

\newcommand{\bb}{\boldsymbol}
\def \d {\mathrm{d}}
\def \e {\mathrm{e}}

\newtheorem{theorem}{Theorem}[section]
\newtheorem{lemma}{Lemma}[section]

\newtheorem{example}{\bf Example}[section]
\theoremstyle{remark}

\numberwithin{equation}{section}

\begin{document}

\title{The Arrow-Hurwicz iteration for virtual element discretizations of the incompressible Navier-Stokes equations}

\author[1]{Binbin Du\thanks{dubb@swufe.edu.cn}}

\author[2]{Shenxiang Cheng\thanks{chengmath192@163.com}}

\author[3]{Yue Yu\thanks{terenceyuyue@xtu.edu.cn}}

\author[2]{Chuanjun Chen\thanks{cjchen@ytu.edu.cn}}

\affil[1]{School of Mathematics, Southwestern University of Finance and Economics, Chengdu 611130, China}
\affil[2]{School of Mathematics and Information Sciences, Yantai University, Yantai 26005, China }
\affil[3]{School of Mathematics and Computational Science, Hunan Key Laboratory for Computation and Simulation in Science and Engineering, Key Laboratory of Intelligent Computing and Information Processing of Ministry of Education, National Center for Applied Mathematics in Hunan, Xiangtan University, Xiangtan, Hunan 411105, China}
	
	\maketitle
	
	\begin{abstract}
     This article presents a detailed analysis of the Arrow-Hurwicz iteration applied to the solution of the incompressible Navier-Stokes equations, discretized by a divergence-free mixed virtual element method. Under a set of appropriate assumptions, it is rigorously demonstrated that the method exhibits geometric convergence, with a contraction factor that remains independent of the mesh sizes. A series of numerical experiments are conducted to validate the theoretical findings and to assess the computational performance of the proposed method.
	\end{abstract}


\section{Introduction}

The incompressible Navier-Stokes (N-S) equations are fundamental in computational fluid dynamics and engineering applications \cite{E-D-W-2005,Quarteroni1994,Temam1984}, modeling phenomena ranging from pipe flow and aerodynamics to weather patterns and biomedical flows. Among various numerical methods, mixed finite element methods have proven particularly effective \cite{Brezzi-Fortin-2013,E-D-W-2005,Girault-Raviart1986,J-R-1982,Temam1984,Du-Huang-2021-NS,Takhirov-2023-NS}, which compute velocity and pressure fields simultaneously. However, these methods typically satisfy the divergence-free constraint only weakly or through projection, especially in three dimensions, which can lead to numerical instabilities in complex flows \cite{Auricchio-Beirao-Lovadina-2010}. This limitation has driven the development of methods that exactly enforce incompressibility, with the virtual element method (VEM) emerging as a powerful approach. VEM naturally enables construction of precisely divergence-free velocity spaces, yielding what we term divergence-free VEM.

The development of VEMs for these problems fundamentally builds upon the VEM framework established for the Stokes equations. Current VEM approaches for the Stokes problem encompass several variants, including either the $H^1$ non-conforming VEMs \cite{Cangiani-Gyrya-Manzini-2016,Liu-Li-Chen-2017,Zhao-Zhang-Mao-2019} or conforming VEMs without divergence \cite{Chen-Wang-2019,Antonietti-Beirao-Mora-2014,Beirao-2016-Hdiv,Beirao-Lovadina-Vacca-2017,Beirao-Dassi-Vacca-2020}, and pseudo-stress velocity formulations \cite{Caceres-Gatica-2017}. Then, these methodologies have been successfully extended in \cite{Liu-Chen-2019,Zhang-Zhao-Li-2023,Wang-Wang-He-2021,Beirao-Lovadina-Vacca-2018,Beirao-Mora-Vacca-2019,Beirao-Dassi-Vacca-2020,Gatica-Munar-2018} to solve the incompressible N-S equations with notable contributions. However, existing analyses primarily focus on spatial discretization, while the error analysis for the resulting nonlinear saddle-point problems remains largely unaddressed, with numerical implementations typically employing fixed-point or Newton iterations as discussed in \cite{Beirao-Mora-Vacca-2019,Beirao-Dassi-Vacca-2020,Zhang-Zhao-Li-2023}.

The mixed finite and virtual element discretizations lead to large-scale nonlinear systems whose efficient solution presents very technical challenges. For finite element approaches, the most widely used iterative strategy involves solving Oseen-type equations at each iteration step, generating nonsymmetric saddle-point systems. To address these nonlinear systems, \cite{He-Li-2009} developed three iterative methods for the stationary N-S equations, combining penalty techniques with Stokes-type, Newton-type, and Oseen-type iterations, respectively. As noted in \cite{Chen-Huang-Sheng-2017}, such discretizations inherently preserve the nonsymmetric nature of the underlying saddle-point systems. For the nonsymmetric linear systems, it is very technical to develop fast solvers, we refer the reader to \cite{Bramble-Pasciak-Vassilev-2000,Cao-Dong-wang-2015,Chen-Huang-Wang-Xu-2015,Wathen-2015} for surveys along this line.

In comparison to the iterative methods proposed in \cite{Cao-Dong-wang-2015,Temam1984,Chen-Huang-Sheng-2017,Bramble-Pasciak-Vassilev-2000,Girault-Raviart1986,Nochetto-Pyo-2004} to address nonsymmetric saddle-point systems, the Arrow-Hurwicz (A-H) method \cite{Arrow-Hurwicz-Uzawa-1958} has several key advantages: (1) flexibility in choosing initial functions, (2) avoidance of saddle-point system solves at each iteration, and (3) an alternating update mechanism for velocity and pressure fields that effectively decouples the system components. Building upon these advantages, this paper presents a combination of divergence-free virtual element methods \cite{Beirao-Lovadina-Vacca-2017,Beirao-Carlo-Giuseppe-2018} with the A-H iteration for solving the incompressible N-S equations. Following the approach established for the finite element methods \cite{Temam1984,Du-Huang-2019,Chen-Huang-Sheng-2017}, we prove that the proposed method achieves geometric convergence with mesh-independent contraction rates, maintaining the favorable convergence properties of the A-H method.

The paper is structured as follows. In Section \ref{sec:NS}, we introduce the N-S equations, some notations and essential theoretical results. Section \ref{sec:VEAH} presents the divergence-free virtual element discretization, and formulates the corresponding A-H iteration scheme for the resulting saddle-point systems, where some basis results for the discrete variational problem are also provided. Section \ref{sec:convergence} establishes a rigorous convergence rate analysis of the A-H method for the divergence-free virtual element discretization. Comprehensive numerical experiments in Section \ref{sec:numerical} validate the theoretical analysis and demonstrate the computational performance of the method, with an extension to unsteady N-S equations discussed in subsection \ref{sec:unsteadyNS}. The implementation details are provided in Appendix \ref{app:computation}. Finally, we end with a short conclusion in Section \ref{sec:conclusion}.

We end this section by introducing some symbols and notations frequently used in this paper.  For a bounded Lipschitz domain $D$, the symbol $( \cdot , \cdot )_D$ denotes the $L^2$-inner product in $D$, $\|\cdot\|_{0,D}$ denotes the $L^2$-norm and $|\cdot|_{s,D}$ is the $H^s(D)$-seminorm. For all integers $k\ge 0$, $\mathbb{P}_k(D)$ is the set of polynomials of degree $\le k$ in $D$. For vector-valued functions or spaces, we use bold symbols, such as $\bb{u}$, $\bb{v}$, $\bb{L}^2(\Omega)$, $\bb{H}^1(\Omega)$, etc.
Moreover, for any two quantities $a$ and $b$, ``$a\lesssim b$" indicates ``$a\le C b$" with the hidden constant $C$ independent of the mesh size, and ``$a\eqsim b$" abbreviates ``$a\lesssim b\lesssim a$".

\section{Steady Navier-Stokes equations} \label{sec:NS}
\label{a1.s2}
\numberwithin{equation}{section}
Let $\Omega \subset \mathbb{R}^2$ be a convex polygonal domain with a Lipschitz continuous boundary $\partial \Omega$. Given an applied body force $\bb{f} \in \bb{L}^2(\Omega)$, the incompressible N-S problem is to find the fluid velocity $\bb{u}$ and pressure $\bb{p}$ such that:
\begin{equation}\label{theNS}
 \begin{cases}
-\nu\Delta \bb{u}+(\bb{u}\cdot \nabla)\bb{u}- \nabla p = \bb{f} &\quad  \mbox{in} \quad  \Omega, \\
\nabla \cdot \bb{u}=0   &\quad  \mbox{in} \quad \Omega, \\
\bb{u}=\bb{0}   & \quad \mbox{on} \quad \partial \Omega.
\end{cases}
\end{equation}
The viscosity coefficient is represented by $\nu= 1/R_e>0$, where $R_e$ denotes the Reynolds number. Since $p$ is unique up to a constant, we assume
\[p\in L^2_0(\Omega):=\Big\{q \in L^2(\Omega):  \int_{\Omega}q \d x=0 \Big\}.\]
In addition, we also introduce the following Sobolev space:
\[\bb{H}^1_0(\Omega):=\{\bb{v}\in\bb{H}^1:\bb{v}=0~\text{on}~\Omega\}.\]


Let $\bb{V} = \bb{H}^1_0(\Omega)$ and $P= L^2_0(\Omega)$. The variational problem regarding Eq.~\eqref{theNS} is: Find $(\bb{u},p)$ $\in$ $\bb{V}$	$\times$ $P$ such that
 \begin{equation}\label{varProblem}
\begin{cases}
 \nu  a(\bb{u},\bb{v})+N(\bb{u};\bb{u},\bb{v}) + b(\bb{v},p) = (\bb{f},\bb{v}), \quad  \bb{v} \in \bb{V}, \\
b\left(\bb{u},q\right) =0, \quad  q \in P, \\
\end{cases}
\end{equation}
where
\[a(\bb{u},\bb{v}) = \int_\Omega \nabla \bb{u}:\nabla \bb{v} \d x, \qquad
b(\bb{v},p)= \int_{\Omega} (\nabla \cdot \bb{v})p \d x , \qquad (\bb{f},\bb{v}) =\int_\Omega \bb{f} \cdot \bb{v} \d x,\]
\[N(\bb{w};\bb{u},\bb{v})=\int_\Omega (\bb{w} \cdot \nabla )\bb{u}\cdot \bb{v} \d x =\int_\Omega (\nabla\bb{u})\bb{w}\cdot \bb{v} \d x.\]
For the bilinear or trilinear forms, as in \cite{Temam1979,Quarteroni1994,Girault-Raviart1986}, we have the following results:
\begin{enumerate}
  \item $a(\cdot,\cdot) $, $b(\cdot,\cdot)$ and $N(\cdot; \cdot,\cdot)$ are continuous,
\begin{align*}
& |a(\bb{u},\bb{v})| \le \|\bb{u}\|_{\bb{V}}\|\bb{v}\|_{\bb{V}}, \quad\bb{u},\bb{v} \in \bb{V},\\
& |b(\bb{v},p)| \le \|\bb{u}\|_{\bb{V} }\| p\|_{P}, \quad \bb{v} \in \bb{V},p \in P,\\
&| N(\bb{w};\bb{u},\bb{v})| \le \mathcal{N} \|\bb{w}\|_{\bb{V}} \|\bb{u}\|_{\bb{V}}\|\bb{v}\|_{\bb{V}},   \quad \bb{w},\bb{u},\bb{v}\in \bb{V},
\end{align*}
  where $\|\bb{v}\|_{\bb{V}} = |\bb{v}|_1$.
  \item $a(\cdot,\cdot)$ is coercive,
\begin{equation*}
	a(\bb{v},\bb{v})\ge \|\bb{v}\|^2_{\bb{V}}, \quad \bb{v} \in \bb{V}.
\end{equation*}
  \item $b(\cdot,\cdot)$ satisfies the inf-sup condition: there exists $\beta>0$ such that
\begin{equation}\label{inf_sup}
 b(\bb{v},q):=\sup_{\bb{0} \neq \bb{v}\in\bb{V}}\frac{b(\bb{v},\bb{q})}{\|\bb{v}\|_{\bb{V}}}\ge \beta \|q\|_P, \quad q \in P.
\end{equation}
\end{enumerate}

According to  \cite{Temam1979,Quarteroni1994,Girault-Raviart1986}, if we assume
\begin{equation}
	\label{define}
	\Lambda := \frac{\mathcal{N} \|\bb{f}\|_{-1} }{\nu^2} <1,
\end{equation}
where
\[
\|\bb{f}\|_{-1}=\sup_{\bb{v}\in \bb{H}^1_0(\Omega)}\frac{|(\bb{f},\bb{v})|}{\|\bb{v}\|_{\bb{V}}}.\]
Then the variational problem \eqref{varProblem} has a unique solution $(\bb{u},p) \in \bb{V}\times P$ satisfying
\[\|\bb{u}\|_{\bb{V}} \leq \frac{\|\bb{f}\|_{-1}}{\nu}.\]

Let us introduce the kernel
\[\bb{Z} = \{ \bb{v}\in \bb{V}:  b(\bb{v}, q) = 0, ~~q \in P \}.\]
By a simple calculation one can find that the bilinear form $N(\bb{u}; \cdot, \cdot): \bb{V} \times \bb{V} \to \mathbb{R}$ is skew-symmetric, i.e.,
\[N(\bb{u}; \bb{v}, \bb{w}) = - N(\bb{u}; \bb{w}, \bb{v}), \quad \bb{v}, \bb{w} \in \bb{V}. \]
Therefore, we can replace $N(\bb{u};\bb{u},\bb{v})$ in \eqref{varProblem} by $\frac12N(\bb{u};\bb{u},\bb{v}) - \frac12 N(\bb{u};\bb{v},\bb{u})$
and introduce the skew-symmetric trilinear form
\[\widetilde{N}(\bb{w};\bb{u},\bb{v}) = \frac12N(\bb{w};\bb{u},\bb{v}) - \frac12 N (\bb{w};\bb{v},\bb{u}).\]

\section{Virtual element discretization with Arrow-Hurwicz iteration}\label{sec:VEAH}

In this section, we apply the divergence-free VEMs proposed in \cite{Beirao-Lovadina-Vacca-2017} to numerically solve the N-S equations. The resulting nonlinear system will be solved by the A-H method \cite{Temam1979}.

\subsection{The virtual element method}

Let $\{\mathcal{T}_h\}$ be a family of decompositions of $\Omega$ into polygonal elements. The generic element is denoted by $K$ with diameter $h_K={\rm diam}(K)$. To avoid complicated presentation, we confine our discussion in two dimensions, with the family of polygonal meshes $\{ \mathcal{T}_h \}_h$ satisfying the following conditions \cite{Beirao-Manzini-2015,Brenner-Guan-Sung-2017,Cangiani-Manzini-Sutton-2017,Beirao-2022-StabilityStokes}:
\begin{itemize}
  \item[\textbf{C1}.]  There exists a real number $\gamma>0$ such that, for each element $K \in \mathcal{T}_h$, it is star-shaped with respect to a disk of radius $\rho_K \ge \gamma h_K$, where $h_K$ is the diameter of $K$.
  \item[\textbf{C2}.] There exists a real number $\gamma_1 > 0$ such that, for each element $K \in \mathcal{T}_h$, the distance between any two vertices of $K$ is $\ge \gamma_1 h_K$.
\end{itemize}
Our analysis can be extended to a more general mesh assumption given in \cite{Brezzi-Buffa-Lipnikov-2009,Chen-HuangJ-2018}.

\subsubsection{The enhancement virtual element spaces}

The divergence-free VEM was first proposed in \cite{Beirao-Lovadina-Vacca-2017} to solve the Stokes problem. For the N-S equations, we need to introduce some projections to approximate the nonlinear term.  For this reason, the standard enhancement technique \cite{Ahmad-Alsaedi-Brezzi-2013,Beirao-Brezzi-Marini-2014} should be used.

Let $k\ge 2$ be an integer. The original virtual element space for the Stokes problem is
\begin{align*}
  \bb{V}_k(K)
  & = \{\bb{v}\in \bb{H}^1(K): \bb{v}|_{\partial K} \in [\mathbb{B}_k (\partial K)]^2, \\
  &  \hspace{3cm}
  \begin{cases}
  -\nu \Delta \bb{v} - \nabla s \in \mathcal{G}^\bot_{k-2}(K), \\
  \text{div}  \bb{v} \in \mathbb{P}_{k-1}(K),
  \end{cases}
  \quad
  \mbox{for some $s \in L^2(K)$}
  \},
\end{align*}
where
\[\mathbb{B}_k(\partial K) = \left\{ v \in C^0(\partial K): v|_e \in \mathbb{P}_k(e),~~ e \subset \partial K \right\},\]
\[\mathcal{G}_{k-2}(K) = \nabla (\mathbb{P}_{k-1}(K)) \subset (\mathbb{P}_{k-2}(K))^2.\]
The basis functions in the orthogonal complement space ~$\mathcal{G}^\bot_{k-2}(K)$ are not easy to select. It can be replaced with
\[
\mathcal{G}^\oplus_{k-2}(K) := \bb{x}^\bot \mathbb{P}_{k-3}(K), \qquad \bb{x}^\bot = (x_2, -x_1)^{\intercal}.
\]

To present the degrees of freedom (DoFs), we introduce a scaled monomial $\mathbb{M}_r(D)$ in a $d$-dimensional domain $D$:
\[
\mathbb  M_{r} (D)= \Big \{ \Big ( \frac{\boldsymbol x -  \boldsymbol x_D}{h_D}\Big )^{\boldsymbol  s}, \quad |\boldsymbol  s|\le r\Big \},
\]
where $h_D$ is the diameter of $D$, $\boldsymbol  x_D$ the centroid of $D$, and $r$ a non-negative integer. For the multi-index ${\boldsymbol{s}} \in {\mathbb{N}^d}$, we follow the usual notation
\[\boldsymbol{x}^{\boldsymbol{s}} = x_1^{s_1} \cdots x_d^{s_d},\quad |\boldsymbol{s}| = s_1 +  \cdots  + s_d.\]

Conventionally, $\mathbb  M_r (D) =\{0\}$ for $r\le -1$. According to the norm equivalence established in \cite{Beirao-2022-StabilityStokes}, the virtual element space can be equipped with the following DoFs:
\begin{itemize}
  \item $\mathbb{D_V}1$: the values of $\bb{v}$ at the vertices of $K$.
  \item $\mathbb{D_V}2$: the values of $\bb{v}$ at $k-1$ distinct internal points of every edge $e\subset \partial K$, with the endpoints consist of $k+1$ Gauss-Lobatto points on $e$.
  \item $\mathbb{D_V}3$: the moments
  \[\frac{1}{|K|} \int_K \bb{v}\cdot \bb{g}_{k-2}^\bot {\mathrm d}x, \quad \bb{g}_{k-2}^\bot \in \mathcal{G}_{k-2}^\oplus(K), \]
  where
  \[\bb{g}_{k-2}^\bot = \bb{x}^\bot m_{k-3}, ~~ m_{k-3} \in \mathbb{M}_{k-3}(K).\]
  \item $\mathbb{D_V}4$: the moments up to order $k-1$ and greater than zero of $\text{div}  \bb{v}$,
  \[\frac{h_K}{|K|}\int_K (\text{div}  \bb{v})m_{k-1} {\mathrm d}x, \quad m_{k-1} \in \mathbb{M}_{k-1}(K)/\mathbb{R}. \]
\end{itemize}

For the introduction of the $L^2$ projection, we introduce a lifting space
\begin{align*}
\widetilde{\bb{V}}_k(K)
  & = \{\bb{v}\in \bb{H}^1(K): \bb{v}|_{\partial K} \in [\mathbb{B}_k (\partial K)]^2, \\
  &  \hspace{3cm}
  \begin{cases}
  -\nu \Delta \bb{v} - \nabla s \in \mathcal{G}_k^\oplus(K), \\
  \text{div}  \bb{v} \in \mathbb{P}_{k-1}(K),
  \end{cases}
  \quad
  \mbox{for some $s \in L^2(K)$}
  \},
\end{align*}
and define on it an elliptic projector $\Pi^K: \widetilde{\bb{V}}_k(K) \to (\mathbb{P}_k(K))^2$, $\bb{v}\mapsto \Pi^K\bb{v}$, satisfying
\begin{equation}\label{ellipticprojStokes}
\begin{cases}
a^K(\Pi^K\bb{v}, \bb{q}) = a^K(\bb{v}, \bb{q}), \quad \bb{q}\in (\mathbb{P}_k(K))^2, \\
P_0^K(\Pi^K\bb{v}) = P_0^K(\bb{v}),
\end{cases}
\end{equation}
where $P_0^K(\bb{v}) = \int_K \bb{v} \d x$.
One can check that $\Pi^K \bb{v}$ can be computed using the DoFs given above. Thus, one can reduce the number of DoFs by employing the enhancement technique \cite{Ahmad-Alsaedi-Brezzi-2013,Beirao-Brezzi-Marini-2014}, with the resulting virtual element space defined as
\begin{equation}\label{VhNS}
\bb{W}_k(K) = \{ \bb{v} \in \widetilde{\bb{V}}_k(K):  (\bb{v} - \Pi^K \bb{v}, \bb{g}_k^\bot)_K = 0, \quad
\bb{g}_k^\bot \in \mathcal{G}_k^\oplus(K) \backslash \mathcal{G}_{k-2}^\oplus(K)\},
\end{equation}
where $\mathcal{G}_k^\oplus(K) \backslash \mathcal{G}_{k-2}^\oplus(K)$ represents the set of polynomials in $\mathcal{G}_k^\oplus(K)$ that are orthogonal to $\mathcal{G}_{k-2}^\oplus(K)$. This enhancement space can be endowed with the same DoFs as those for $\bb{V}_k(K)$.

\subsubsection{The discrete problem and basic results}

Let $\bb{u}_h, \bb{v}_h, \bb{w}_h \in \bb{W}_h$. We introduce the virtual element discretizations as follows.
\begin{itemize}
  \item The bilinear form $a^K(\cdot, \cdot)$ is approximated as
\[a_h^K(\bb{u}_h, \bb{v}_h) = a^K(\Pi^K \bb{u}_h, \Pi^K \bb{v}_h) + S^K(\bb{u}_h - \Pi^K \bb{u}_h, \bb{v}_h-\Pi^K \bb{v}_h),\]
where
\[S^K(\bb{u}_h , \bb{v}_h) = \sum_{i=1}^{N_K} \chi_i(\bb{u}_h) \chi_i(\bb{v}_h),\]
with $\chi_i$ being the DoFs of $\bb{W}_k(K)$ and $N_K$ the number of the local DoFs. According to the norm equivalence in \cite{Beirao-2022-StabilityStokes}, there exist positive constant $\alpha_*$ and $\alpha^*$, independent of the element $K$, such that
\begin{equation}\label{normeq}
\alpha_* a^K(\bb{v}_h, \bb{v}_h) \le a_h^K(\bb{v}_h, \bb{v}_h) \le \alpha^* a^K(\bb{v}_h, \bb{v}_h).
\end{equation}

  \item For the approximation of the local trialinear form $N(\cdot;\cdot,\cdot)$, we define
  \[\Pi_{k-1}^0\nabla: \bb{W}_k(K) \to (\mathbb{P}_{k-1}(K))^{2\times 2}, \quad \bb{v}_h \mapsto \Pi_{k-1}^0\nabla \bb{v}_h,\]
  which satisfies
  \[(\Pi_{k-1}^0\nabla \bb{v}_h, \bb{P})_K = (\nabla \bb{v}_h, \bb{P})_K, \qquad \bb{P} \in (\mathbb{P}_{k-1}(K))^{2\times 2}.\]
  Then $N(\bb{w}_h;\bb{u}_h,\bb{v}_h)$ is approximated as
  \[N_h(\bb{w}_h; \bb{u}_h, \bb{v}_h) = \sum\limits_{K\in \mathcal{T}_h}\int_K [(\Pi_{k-1}^0 \nabla \bb{u}_h) \Pi_k^0 \bb{w}_h] \cdot \Pi_k^0 \bb{v}_h \d x,\]
  where $\Pi_k^0$ is the standard $L^2$ projector.
  This approximated trilinear form is continuous, i.e., there exists a uniformly bounded with respect to $h$ such that
  \begin{equation} \label{Nhbound}
  | N_h(\bb{w}; \bb{u}, \bb{v}) | \le  \widetilde{\mathcal{N}} \|\bb{w}\|_{\bb{V}} \|\bb{u}\|_{\bb{V}}\|\bb{v}\|_{\bb{V}},   \quad \bb{w},\bb{u},\bb{v}\in \bb{V}.
  \end{equation}

  \item The load term $(\bb{f}, \bb{v}_h)$ can be approximated by
  \[(\bb{f}_h, \bb{v}_h ):= ( \Pi_k^0  \bb{f}, \bb{v}_h) = ( \bb{f}, \Pi_k^0 \bb{v}_h).\]
\end{itemize}

The discrete mixed variational problem introduced in \cite{Beirao-Carlo-Giuseppe-2018} is to find $(\bb{u}_h, \bb{p}_h) \in \bb{W}_h \times P_h$ such that
\begin{equation}\label{NSVEM}
 \begin{cases}
  \nu  a_h(\bb{u}_h,\bb{v}_h)+  \widetilde{N}_h(\bb{u}_h;\bb{u}_h,\bb{v}_h)  + b (\bb{v}_h,p_h) = ( \bb{f}_h, \bb{v}_h ), \quad \bb{v}_h \in \bb{W}_h,  \\
  b(\bb{u}_h,q_h) = 0, \quad  q_h \in P_h,
\end{cases}
\end{equation}
where $P_h$ is a piecewise polynomial space on $\mathcal{T}_h$ with degree $\le k-1$ and
\[\widetilde{N}_h(\bb{w}_h;\bb{u}_h,\bb{v}_h) = \frac12 N_h(\bb{w}_h;\bb{u}_h,\bb{v}_h) - \frac12 N_h(\bb{w}_h;\bb{v}_h,\bb{h}_h).\]

For later uses, we list some basic results for the above discrete method.
The first fundamental result is the discrete inf-sup condition.

\begin{lemma}[inf-sup condition, Proposition 3.4 in \cite{Beirao-Carlo-Giuseppe-2018}] \label{lem:infsup}
There exists a positive $\widetilde{\beta}$, independent of $h$, such that
\[\sup_{\bb{0}\ne \bb{v}_h \in \bb{W}_h} \frac{b(\bb{v}_h, q_h)}{\|\bb{v}_h\|_{\bb{V}}} \ge \widetilde{\beta} \|q_h\|_0, \qquad q_h \in P_h.\]
\end{lemma}

This together with the stability in \eqref{normeq} yields the well-posedness of the discrete method.

\begin{lemma}[Theorem 3.5 in \cite{Beirao-Carlo-Giuseppe-2018}] \label{lem:wellposedness}
Suppose that
\[\widetilde{\Lambda}  = \frac{\widetilde{\mathcal{N}} \|\bb{f}_h\|_{-1} }{(\alpha_*\nu)^2}<1, \]
where $\alpha_*$ and $\widetilde{\mathcal{N}}$ are defined in \eqref{normeq} and \eqref{Nhbound}. Then the problem \eqref{NSVEM} has a unique solution $(\bb{u}_h, \bb{p}_h) \in \bb{W}_h \times P_h$ satisfying
\[|\bb{u}_h|_1 \le \frac{\|\bb{f}_h\|_{-1}}{\alpha_* \nu}.\]
\end{lemma}

Combining \eqref{NSVEM} and Lemmas \ref{lem:infsup} and \ref{lem:wellposedness}, we conclude that
\begin{align}
\widetilde{\beta} \|p_h\|_0
& \le \sup_{\bb{0}\ne \bb{v}_h \in \bb{W}_h} \frac{( \bb{f}_h, \bb{v}_h ) -\nu a_h(\bb{u}_h,\bb{v}_h)-\widetilde{N}_h(\bb{u}_h;\bb{u}_h,\bb{v}_h)}{\|\bb{v}_h\|_{\bb{V}}} \nonumber \\
& \le \|\bb{f}_h\|_{-1} + \nu \alpha^* \|\bb{u}_h\|_{\bb{V}} + \widetilde{\mathcal{N}} \|\bb{u}_h\|_{\bb{V}}^2
\le \|\bb{f}_h\|_{-1} +  \frac{\alpha^*}{\alpha_* } \|\bb{f}_h\|_{-1} + \widetilde{\mathcal{N}} \frac{\|\bb{f}_h\|_{-1}^2}{\alpha_*^2 \nu^2} \nonumber \\
& = ( 1 + \frac{\alpha^*}{\alpha_* }+ \widetilde{\Lambda}  ) \|\bb{f}_h\|_{-1}. \label{phbound}
\end{align}

\subsection{The Arrow-Hurwicz iteration}

The A-H method offers flexibility in selecting initial functions and eliminates the need to solve saddle-point systems at each iteration.
Let $\rho$ and $\alpha$ be positive parameters. The A-H method for our discrete problem can be described as follows:
\begin{enumerate}
  \item[Step~1:]   Use the solution of the problem without nonlinear terms, specifically the Stokes problem, as the initial guess for the iteration: Find $(\bb{u}_h^0, \bb{p}_h^0) \in \bb{W}_h \times P_h$ such that
  \begin{equation}\label{NSAH0}
 \begin{cases}
  \nu a_h(\bb{u}_h^0,\bb{v}_h) + b (\bb{v}_h,p_h^0) = ( \bb{f}_h, \bb{v}_h ), \quad \bb{v}_h \in \bb{W}_h,  \\
  b(\bb{u}_h^0,q_h) = 0, \quad  q_h \in P_h.
\end{cases}
\end{equation}
  \item[Step~2:] For $n=0,1,\cdots, m-1$, given $(\bb{u}_h^n, \bb{p}_h^n) \in \bb{W}_h \times Q_h$, we sequentially find $(\bb{u}_h^{n+1}, \bb{p}_h^{n+1}) \in \bb{W}_h \times P_h$ such that
  \begin{equation}\label{NSAHn}
 \begin{cases}
  \rho^{-1} a_h(\bb{u}_h^{n+1} - \bb{u}_h^n,\bb{v}_h) +  \nu  a_h(\bb{u}_h^n,\bb{v}_h) + \widetilde{N}_h(\bb{u}_h^n;\bb{u}_h^{n+1},\bb{v}_h) \\
  \hspace{1cm}  =  -b (\bb{v}_h,p_h^n) +  ( \bb{f}_h, \bb{v}_h ), \quad \bb{v}_h \in \bb{W}_h,  \\
  \alpha (p_h^{n+1}-p_h^n, q_h) - \rho b(\bb{u}_h^{n+1},q_h) = 0, \quad  q_h \in P_h.
\end{cases}
\end{equation}
\end{enumerate}

Since the constraint $\int_\Omega p_h{\mathrm d}x$ is not naturally imposed, we introduce a Lagrange multiplier $\lambda^{n+1} \in \mathbb{R}$. The augmented variational formulation for \eqref{NSAHn} is to find $(\bb{u}_h^{n+1}, \bb{p}_h^{n+1}, \lambda^{n+1}) \in \bb{W}_h \times P_h \times \mathbb{R}$ such that
  \begin{equation}\label{NSAHnLagrange}
 \begin{cases}
  \rho^{-1} a_h(\bb{u}_h^{n+1} - \bb{u}_h^n,\bb{v}_h) +  \nu  a_h(\bb{u}_h^n,\bb{v}_h) + \widetilde{N}_h(\bb{u}_h^n;\bb{u}_h^{n+1},\bb{v}_h) \\
  \hspace{1cm}   =  -b (\bb{v}_h,p_h^n) +  ( \bb{f}_h, \bb{v}_h ), \quad \bb{v}_h \in \bb{W}_h,  \\
  \alpha (p_h^{n+1}-p_h^n, q_h) - \rho b(\bb{u}_h^{n+1},q_h) + \displaystyle \lambda^{n+1} \int_\Omega q_h{\mathrm d}x = 0, \quad  q_h \in P_h,\\
  \displaystyle \mu\int_\Omega p_h{\mathrm d}x  = 0, \quad \mu \in \mathbb{R}.
\end{cases}
\end{equation}

Let the basis of $\bb{V}_h$ be $\bb{\phi}_i$, where $i=1,\cdots, N$, and $N$ is the dimension of $\bb{V}_h$. Then,
\[
\bb{u}_h^{n+1} = \sum\limits_{i=1}^N \chi_i^{n+1} \bb{\phi}_i \equiv \bb{\phi}^{\intercal}\bb{\chi}^{n+1}.
\]
Similarly, let the basis of $Q_h$ be $\psi_l$, where $l=1,\cdots,M$, i.e.,
\[
p_h^{n+1} = \sum\limits_{l=1}^M p_l^{n+1} \psi_l.
\]
Substitute these expansions into \eqref{NSAH0} and \eqref{NSAHn} or \eqref{NSAHnLagrange}, and set $\bb{v}_h = \bb{\phi}_j$, $q_h = \psi_l$, one can derive the associated matrix form:
\begin{enumerate}
  \item[Step~1:] Find $(\bb{\chi}^0, \bb{p}^0)$ satisfying
 \[
\begin{bmatrix}
\nu A   & B & \bb{0}\\
B^{\intercal} & O & \bb{d}\\
\bb{0}^{\intercal} & \bb{d}^{\intercal}  & 0
\end{bmatrix}
\begin{bmatrix}
\bb{\chi}^0 \\
\bb{p}^0 \\
\lambda^0
\end{bmatrix}
=\begin{bmatrix}
\bb{F} \\
\bb{0} \\
0
\end{bmatrix},
\]
where
\[A_{ij}=a_h(\bb{\phi}_j,\bb{\phi}_i), \quad  B_{jl} = b(\bb{\phi}_j, \psi_l),\quad \bb{F}_j = (\bb{f}, \Pi_k^0\bb{\phi}_j).\]

  \item[Step~2:] For $n=0,1,\cdots, m-1$, given $(\bb{\chi}^n, \bb{p}^n)$, we get $(\bb{\chi}^{n+1}, \bb{p}^{n+1})$ by solving
\begin{equation}\label{iteration}
\begin{bmatrix}
\rho^{-1} A + N_2^n   & O & \bb{0}\\
-\rho B^{\intercal} & \alpha C & \bb{d}\\
\bb{0}^{\intercal} & \bb{d}^{\intercal}  & 0
\end{bmatrix}
\begin{bmatrix}
\bb{\chi}^{n+1} \\
\bb{p}^{n+1} \\
\lambda^{n+1}
\end{bmatrix}
=\begin{bmatrix}
(\rho^{-1}-\nu) A  \bb{\chi}^n - B \bb{p}^n + \bb{F} \\
\alpha C \bb{p}^n \\
0
\end{bmatrix},
\end{equation}
where $C_{l, l'} = ( \psi_{l'}, \psi_{l})$ and
\[(N_2^n)^K_{ij} = \frac12 N_h(\bb{u}_h^n; \bb{\phi}_j, \bb{\phi}_i) - \frac12 N_h(\bb{u}_h^n; \bb{\phi}_i, \bb{\phi}_j), \]
 \[ N_h(\bb{u}_h^n; \bb{\phi}_j, \bb{\phi}_i) = \int_K [(\Pi_{k-1}^0 \nabla\bb{\phi}_j) \Pi_k^0 \bb{u}_h^n] \cdot \Pi_k^0 \bb{\phi}_i \d x .\]
\end{enumerate}
 We note that at each iteration of \eqref{iteration}, we can first compute $\bb{\chi}^{n+1}$ and then substitute it into the equation for $\bb{p}^{n+1}$, thereby eliminating the need to solve saddle-point systems at each iteration.

For future reference, we establish the error bound between the variational problems derived from the N-S equations and the Stokes equations.
\begin{lemma}\label{lem:NSS}
	 Let $(\bb{u}^0_h, p^0_h)$ be the solution to \eqref{NSAH0} and let $(\bb{u}_h, p_h)$ be the solution to \eqref{NSVEM}. Then there holds
\[
|\bb{u}_h - \bb{u}^0_h|_1 \leq \widetilde{\Lambda}  (\alpha_*\nu)^{-1} \| \bb{f} \|_{-1}, \qquad
\| p_h - p^0_h \|_0 \le 2 \widetilde{\Lambda}  \widetilde{\beta}^{-1} \| \bb{f} \|_{-1}.
\]
\end{lemma}
\begin{proof}
Let $\bb{E}_h^0 = \bb{u}_h - \bb{u}^0_h$ and $e_h^0 = p_h - p^0_h$. Subtracting \eqref{NSVEM} from \eqref{NSAH0} yields
\[\nu a_h ( \bb{E}_h^0,  \bb{v}_h ) + ({\rm div}\bb{v}_h, e^0_h) = - \widetilde{N}_h(\bb{u}_h;\bb{u}_h,\bb{v}_h).\]
Take $\bb{v}_h = \bb{E}_h^0$ to derive
\[\nu a_h ( \bb{E}_h^0,  \bb{E}_h^0 ) + ({\rm div}\bb{E}_h^0, e^0_h) = - \widetilde{N}_h(\bb{u}_h;\bb{u}_h,\bb{E}_h^0).\]
The divergence-free conditions in the variational problem imply $({\rm div}\bb{E}_h^0, e^0_h) = 0$. Therefore, by Lemma \ref{lem:wellposedness} and the stability in \eqref{normeq}, one has
\[	\alpha_* \nu | \bb{E}^0_h |^2_1  \le \widetilde{\mathcal{N}} | \bb{u}_h|^2_1 |\bb{E}^0_h|_1 \le \widetilde{\Lambda}  \|\bb{f}\|_{-1}|\bb{E}^0_h|_1,\]
which gives the first inequality.

For the second inequality, according to the inf-sup condition \eqref{lem:infsup}, we obtain
\begin{align*}
\widetilde{\beta} \| e_h \|_0
&\le  \sup\limits_{\bb{v}_h \in \bb{W}_h}\frac{({\rm div}\bb{v}_h, e^0_h )}{| \bb{v}_h |_1}\\
&  = \sup\limits_{\bb{v}_h \in V_h} \frac{- \widetilde{N}_h(\bb{u}_h;\bb{u}_h,\bb{v}_h) - \nu a_h( \bb{E}_h^0,  \bb{v}_h )}{| \bb{v}_h |_1}\\
&\le \alpha_* \nu |\bb{E}^0_h|_1 + \widetilde{\mathcal{N}} |\bb{u}_h |^2_1\\
&\le 2 \widetilde{\Lambda}  \|\bb{f}\|_{-1}.
\end{align*}
The proof is completed.
\end{proof}

\section{Convergence analysis of the Arrow-Hurwicz method} \label{sec:convergence}

For the convergence analysis, we need to establish the boundedness of the iterative sequence for the A-H method \eqref{NSAHn}.

\begin{lemma}\label{lem:boundedness}
Let $(\bb{u}_h, p_h) \in \bb{W}_h \times P_h$ be the solution of the discrete variational problem \eqref{NSVEM}, and let $\{(\bb{u}_h^n, p_h^n)\}$ be the function sequence generated by the A-H iteration \eqref{NSAHn}. Suppose that $\widetilde{\Lambda}  < 1$ and
\begin{equation}\label{convergeCondition}
 \alpha^* |1-\rho \nu| + \rho (\alpha_* \nu) \widetilde{\Lambda}  + \frac{\rho^2}{2\alpha} < \alpha_*.
\end{equation}
 Then the sequences $\{|\bb{u}_h^n|_1\}$ and $\{\|p_h^n\|_0\}$ are uniformly bounded with respect to $h$.
\end{lemma}

\begin{proof}
Let $\bb{E}_h^n = \bb{u}_h^n - \bb{u}_h$ and $e_h^n = p_h^n - p_h$. Owing to Lemma \ref{lem:wellposedness} and the boundedness in \eqref{phbound}, it reduces to examine the boundedness of the sequences $|\bb{E}_h^n|_1$ and $\|e_h^n\|_0$.
By subtracting \eqref{NSAHn} from \eqref{NSVEM}, we obtain
\begin{equation}\label{eq:Subtractresult}
    \rho^{-1} a_h(\bb{E}_h^{n+1} - \bb{E}_h^n,\bb{v}_h)
+ \nu a_h(\bb{E}_h^n,\bb{v}_h) + b (\bb{v}_h,e_h^n)
  = - \widetilde{N}_h(\bb{E}_h^n;\bb{u}_h,\bb{v}_h) - \widetilde{N}_h(\bb{u}_h^n;\bb{E}_h^{n+1},\bb{v}_h).
\end{equation}
Taking $\bb{v}_h = \bb{E}_h^{n+1}$ in \eqref{eq:Subtractresult} and using $\widetilde{N}_h(\bb{u}_h^n;\bb{E}_h^{n+1},\bb{E}_h^{n+1}) = 0$, we get
\[\rho^{-1}a_h( \bb{E}^{n+1}_h- \bb{E}^{n}_{h},\bb{E}^{n+1}_{h})+ \nu a_h ( \bb{E}^{n}_h, \bb{E}^{n+1}_{h})+ (\text{div} \bb{E}_h^{n+1}, e^n_h)= - \widetilde{N}_h(\bb{E}_h^n;\bb{u}_h,\bb{E}_h^{n+1}).\]
Now, due to the stability in \eqref{normeq}, we have
\begin{equation}\label{errEq}
\alpha_* | \bb{E}_h^{n+1}|_1^2 + \rho(\text{div} \bb{E}^{n+1}_h, e^n_h) \le |1-\rho \nu| \alpha^*  |\bb{E}^n_h|_1 | \bb{E}^{n+1}_h|_1 -\rho \widetilde{N}_h(\bb{E}^n_h;\bb{u}_h,\bb{E}^{n+1}_h).
\end{equation}

On the one hand, using the second equalities in \eqref{NSVEM} and \eqref{NSAHn}, which gives
\begin{equation}\label{rhodivqh}
\alpha(e^{n+1}_h-e^{n}_h, q_h) = \rho(\text{div}\bb{E}^{n+1}_h, q_h), \qquad q\in P_h.
\end{equation}
Taking $q_h = e_h^n$, we express the second term on the left-hand side in terms of $\{e_h^n\}$:
\begin{equation}\label{divEhn1}
\rho(\text{div}\bb{E}^{n+1}_h, e^n_h)=\alpha(e^{n+1}_h-e^{n}_h,e^{n}_h)=\frac{\alpha}{2}( \|  e^{n+1}_{h} \| ^2_0- \|  e^n_h  \| ^2_0- \|  e^{n+1}_h -e^n_h \| _0^2).
\end{equation}
Setting $q _h = e_h^{n+1} - e_h^n$ in \eqref{rhodivqh}, giving
\[\|  e^{n+1}_h-e^{n}_h \| _0 \leq \rho\alpha^{-1} |  \bb{E}^{n+1}_h | _1.\]

Further, inserting \eqref{divEhn1} into \eqref{errEq} yields
\begin{align*}
& 2 \alpha_* | \bb{E}^{n+1}_h  | _1^2+\alpha  \|  e^{n+1}_h \| ^2_0 \\
& \le \alpha \|  e^n_h \| ^2_0 + \alpha  \|  e^{n+1}_h -e^{n}_h \| ^2_0+2|1-\rho \nu| \alpha^* |\bb{E}^n_h|_1|\bb{E}^{n+1}_h|_1-2\rho \widetilde{N}_h(\bb{E}^n_{h};\bb{u}_h,\bb{E}^{n+1}_h) \\
& \le \alpha \|  e^n_h \| ^2_0 + \rho^2 \alpha^{-1} |  \bb{E}^{n+1}_h |_1^2 + 2|1-\rho \nu| \alpha^* |\bb{E}^n_h|_1 |\bb{E}^{n+1}_h|_1 + 2\rho \widetilde{\mathcal{N}} |\bb{u}_h|_1 |\bb{E}^n_{h}|_1  |\bb{E}^{n+1}_h|_1.
 \end{align*}
Applying Lemma \ref{lem:wellposedness} and the Young's inequality $2ab<\varepsilon a^2+\varepsilon^{-1}b^2$, we have
\begin{align*}
& 2\alpha_* | \bb{E}^{n+1}_h  | _1^2+\alpha  \|  e^{n+1}_h \| ^2_0 \\
& \le \alpha \|  e^n_h \| ^2_0 + \rho^2 \alpha^{-1} |  \bb{E}^{n+1}_h |_1^2 + 2[ |1-\rho \nu| \alpha^* + \rho (\alpha_* \nu) \widetilde{\Lambda} ] |\bb{E}^n_h|_1 |\bb{E}^{n+1}_h|_1  \\
& \le \alpha \|  e^n_h \| ^2_0 + \rho^2 \alpha^{-1} |  \bb{E}^{n+1}_h |_1^2 + [ |1-\rho \nu|\alpha^* + \rho (\alpha_* \nu) \widetilde{\Lambda} ](\varepsilon |  \bb{E}^{n+1}_h | ^2_1 +\varepsilon^{-1}  |  \bb{E}^{n}_{h} | ^2_1  ),
\end{align*}
where $\varepsilon>0$ is a constant to be determined.
Let $r = \alpha^* |1-\rho \nu| + \rho (\alpha_* \nu) \widetilde{\Lambda} $. The estimate can be rewritten as
\begin{equation}\label{valueeq:convergence}
(2\alpha_* -\rho^2 \alpha^{-1} - r \varepsilon) | \bb{E}^{n+1}_h  | _1^2 + \alpha  \|  e^{n+1}_h \| ^2_0 \le  r \varepsilon^{-1} |  \bb{E}^{n}_{h} | ^2_1 + \alpha \|  e^n_h \| ^2_0.
\end{equation}

On the other hand, we try to find $\varepsilon>0$ such that
\[2\alpha_* -\rho^2 \alpha^{-1} - r \varepsilon = r \varepsilon^{-1}.\]
That is, $\varepsilon$ is the positive root of the quadratic equation
\[r \varepsilon^2 - (2\alpha_* -\rho^2 \alpha^{-1}) \varepsilon + r = 0.\]
The discriminant of the equation is
\begin{equation}\label{discriminant}
\Delta = (2\alpha_* -\rho^2 \alpha^{-1})^2 - 4 r^2
= (2\alpha_* -\rho^2 \alpha^{-1} + 2r)(2\alpha_* -\rho^2 \alpha^{-1} - 2r),
\end{equation}
which is positive when
\[2\alpha_* -\rho^2 \alpha^{-1} - 2r > 0 \qquad \mbox{or} \qquad
\rho^2 /(2\alpha) + r < \alpha_*.\]

In this case, we set
\[\varepsilon = \varepsilon^* = \frac{2\alpha_* -\rho^2 \alpha^{-1} + \sqrt{\Delta} }{2r} ,\]
which results in
\[
\mathcal{D} | \bb{E}^{n+1}_h  | _1^2 + \alpha  \|  e^{n+1}_h \| ^2_0
 \le \mathcal{D} | \bb{E}^n_h  | _1^2 + \alpha  \|  e^n_h \| ^2_0,
\]
with
\[\mathcal{D} = r /\varepsilon^* = \alpha_* - \frac{\rho^2}{2\alpha} + \sqrt{(\alpha_* - \frac{\rho^2}{2\alpha})^2 - r^2} > 0. \]
The proof is completed by combining Lemma \ref{lem:NSS}.
\end{proof}

With the help of the boundedness of sequences $\{|\bb{u}_h^n|_1\}$ and $\{\|p_h^n\|_0\}$, we are in a position to establish the convergence of the A-H based method.

\begin{theorem} Let $(\bb{u}_h, p_h) \in \bb{W}_h \times P_h$ be the solution of the discrete variational problem \eqref{NSVEM} and $\{(\bb{u}_h^n, p_h^n)\}$ be the function sequence generated by the A-H iteration \eqref{NSAHn}. If $\widetilde{\Lambda} <1$ and the condition \eqref{convergeCondition} holds, then we have the following estimate
\[\mathcal{F}  |  \bb{E}^{n+1}_h | ^2_1+ \alpha \|  e^{n+1}_h \| ^2_0 \leq \gamma (\mathcal{F} |  \bb{E}^n_h | ^2_1 +\alpha \|  e^n_h \| ^2_0 ),\]
where $\mathcal{F} \in (0, \alpha_*)$ and $\gamma \in(0,1)$ are two generic constants independent of $n$ and $h$.
\end{theorem}
\begin{proof}
From Lemma \ref{lem:boundedness}, we conclude that there is a positive constant $\mathcal{F}_1$,  independent of $n$ and $h$, such that $|\bb{u}^n_h |_1  \le \mathcal{F}_1$. Since the inf-sup condition in Lemma \ref{lem:infsup}, Eq.~\eqref{eq:Subtractresult}, the stability \eqref{normeq} and Lemma \ref{lem:wellposedness}, we obtain
\begin{align*}
\widetilde{\beta}  \|  e^n_h \|_0
& \le \sup_{\bb{v}_h \in \bb{W}_h} \frac{b(\bb{v}_h, e_h^n )}{  |  \bb{v}_h| _1}\\
& \le \nu   \alpha^*  |\bb{E}^n_h|_1 + \rho^{-1} \alpha^* |\bb{E}^{n+1}_h-\bb{E}^{n}_h|_1  + \widetilde{\mathcal{N}}  |  \bb{E}^n_{h} | _1 |  \bb{u}_h | _1+\widetilde{\mathcal{N}} |  \bb{u}^n_h | _1 |\bb{E}_h^{n+1} | _1\\
& \le \nu \alpha^*| \bb{E}^n_h| _1  + \rho^{-1} \alpha^* |\bb{E}^{n+1}_{h} |_1 + \rho^{-1} \alpha^*|\bb{E}^{n}_{h}|_1 + \alpha_*\nu \widetilde{\Lambda}   | \bb{E}^n_h |_1+\widetilde{\mathcal{N}}\mathcal{F}_1 |  \bb{E}^{n+1}_{h} | _1\\
& = (\alpha^* \nu+\alpha_*\nu  \widetilde{\Lambda} + \rho^{-1} \alpha^*) |  \bb{E}^{n}_h | _1+ (\widetilde{\mathcal{N}}\mathcal{F}_1 +\rho^{-1}\alpha^*) |  \bb{E}^{n+1}_h  | _1.
\end{align*}
Squaring the above equation and applying the basic inequality $(u+v)^2\le 2(u^2+v^2)$ for any real numbers $u$ and $v$, we arrive at
\[\widetilde{\beta}^2 \|  e^n_h \| ^2_0\leq 2(\alpha^* \nu+\alpha_*\nu  \widetilde{\Lambda} + \rho^{-1} \alpha^*)^2 |  \bb{E}^n_h  | ^2_1 + 2 (\widetilde{\mathcal{N}}\mathcal{F}_1 +\rho^{-1}\alpha^*)^2 |  \bb{E}^{n+1}_h | ^2_1,\]
which implies
\begin{equation}	\label{imply.eq}
|\bb{E}^{n+1}_h |^2_1 \ge \mathcal{F}_2 \|  e^n_h \| ^2_0 - \mathcal{F}_3 |  \bb{E}^n_h | ^2_1,
\end{equation}
where
\[
\mathcal{F}_2 = \frac{1}{2}\Big(\frac{\widetilde{\beta}}{\widetilde{\mathcal{N}}\mathcal{F}_1 +\rho^{-1}\alpha^*}\Big)^2,  \qquad
\mathcal{F}_3 = \Big( \frac{\alpha^* \nu+\alpha_*\nu  \widetilde{\Lambda} + \rho^{-1} \alpha^*}{\widetilde{\mathcal{N}}\mathcal{F}_1 +\rho^{-1}\alpha^*}\Big)^2. 	\]

Rewrite \eqref{valueeq:convergence} as
\begin{equation*}
    \delta |  \bb{E}^{n+1}_{h} | ^2_1 +
    (2\alpha_* -\rho^2 \alpha^{-1} - r \varepsilon-\delta) | \bb{E}^{n+1}_h  | _1^2 + \alpha  \|  e^{n+1}_h \| ^2_0 \le  r \varepsilon^{-1} |  \bb{E}^{n}_{h} | ^2_1 + \alpha \|  e^n_h \| ^2_0.
\end{equation*}
where $\varepsilon, \delta > 0$ are parameters to be determined. This together with \eqref{imply.eq} yields
\begin{equation}\label{rateeq}
(2\alpha_* -\rho^2 \alpha^{-1} - r \varepsilon-\delta) |  \bb{E}^{n+1}_h | ^2_1+\alpha \|  e^{n+1}_{h} \|  ^2_0\leq (r\varepsilon^{-1}+\delta \mathcal{F}_3) |  \bb{E}^n_h | ^2_1 + (\alpha-\delta \mathcal{F}_2) \|  e^n_h \| ^2_0.
\end{equation}
To derive the desired estimate, we choose $\varepsilon, \delta>0$ such that
\begin{align}	
& \frac{2\alpha_* -\rho^2 \alpha^{-1} - r \varepsilon-\delta}{\alpha} = \frac{r\varepsilon^{-1}+\delta \mathcal{F}_3}{\alpha-\delta\mathcal{F}_2}, \label{eq5} \\
& 2\alpha_*-\rho^2 \alpha^{-1}-r\varepsilon-\delta>0,   \qquad  \alpha-\delta \mathcal{F}_2>0.  \label{ineqdelta}
\end{align}

Furthermore, we can express \eqref{eq5} as an quadratic equation in terms of $\delta$, i.e.,
\begin{equation} \label{eq6}
\mathcal{F}_2\delta^2-[\alpha  +\alpha\mathcal{F}_3+\mathcal{F}_2(2\alpha_*- \rho^2\alpha^{-1}-r\varepsilon)]\delta+\alpha(2\alpha_*-\rho^2\alpha^{-1}-r\varepsilon-r\varepsilon^{-1})=0.
\end{equation}
Under the condition of \eqref{ineqdelta}, one can find that
\begin{align*}
\alpha(2\alpha_*-\rho^2\alpha^{-1}-r\varepsilon-r\varepsilon^{-1})
& = [\alpha  +\alpha\mathcal{F}_3+\mathcal{F}_2(2\alpha_*- \rho^2\alpha^{-1}-r\varepsilon)]\delta - \mathcal{F}_2\delta^2 \\
& \ge (\alpha  +\alpha\mathcal{F}_3+ \delta \mathcal{F}_2)\delta - \mathcal{F}_2\delta^2\\
&>0
\end{align*}
or
\[ -r\varepsilon^2 + (2\alpha_*-\rho^2\alpha^{-1}) \varepsilon -r > 0.\]

In view of \eqref{discriminant}, the quadratic function on the left-hand side has two distinct roots under the condition of \eqref{convergeCondition}, given as
\[\varepsilon_1 = \frac{(2\alpha_*-\rho^2\alpha^{-1}) -\sqrt{\Delta} }{2r}, \qquad
\varepsilon_2 = \frac{(2\alpha_*-\rho^2\alpha^{-1}) + \sqrt{\Delta} }{2r}.\]
This implies that $\varepsilon$ must satisfy $\varepsilon_1 < \varepsilon< \varepsilon_2$ and we set
\[\varepsilon = \varepsilon^*:= \frac{2\alpha_*-\rho^2\alpha^{-1} }{2r} = \frac{\alpha_*-\rho^2/(2\alpha)  }{r}.\]
Thus, Eq.~\eqref{eq6} becomes
\begin{equation}
	a\delta^2-b\delta+c=0, \nonumber
\end{equation}
where
\[
a=\mathcal{F}_2, \qquad b=\alpha+\alpha\mathcal{F}_3+\mathcal{F}_2s, \qquad c=\alpha (s-\frac{r^2}{s}), \qquad
	s:=\alpha_*  -\frac{\rho^2}{2\alpha}.
\]

Since $a, b, c>0$ and
\[
b^2-4ac>(\alpha+\mathcal{F}_2s)^2-4\alpha \mathcal{F}_2s\ge 0,
\]
the equation in Eq.~\eqref{eq6} has two positive roots. So we can take
\[\delta=\delta^*:=\frac{b-\sqrt{b^2-4ac}}{2a}.\]

With the parameters $\varepsilon$ and $\delta$ chosen as above, we plug them into \eqref{rateeq} and use \eqref{eq5} to obtain
\[
(2\alpha_* -\rho^2 \alpha^{-1} - r \varepsilon-\delta) |  \bb{E}^{n+1}_h | ^2_1+\alpha \|  e^{n+1}_{h} \|  ^2_0\leq (r\varepsilon^{-1}+\delta \mathcal{F}_3) |  \bb{E}^n_h | ^2_1 + (\alpha-\delta \mathcal{F}_2) \|  e^n_h \| ^2_0.
\]
\[\mathcal{F} |  \bb{E}^{n+1}_h | ^2_1+\alpha \|  e^{n+1}_h \| ^2_0 \leq \gamma(\mathcal{F} |  \bb{E}^n_h | ^2_1+\alpha \|  e^n_h \| ^2_0), \]
where
\[\mathcal{F} =\alpha_*-\frac{\rho^2}{2\alpha}-\delta^* = s - \delta^*, \qquad \gamma=1-\alpha^{-1}\delta^*\mathcal{F}_2.\]
It is apparent that $\mathcal{F}<\alpha_*$ and $\gamma<1$. Following the same discussion in the proof of \cite[Theorem 3.1]{Chen-Huang-Sheng-2017},  we can also derive $\mathcal{F}, \gamma>0$. This completes the proof.
\end{proof}

Let $\widetilde{\Lambda}<1$ and $\alpha_* \le \alpha^*$. We can rewrite condition \eqref{convergeCondition} as follows:
\begin{itemize}
  \item If $\rho \le \frac{1}{\nu}$,  Eq.~\eqref{convergeCondition} can be written as
\[ \frac{\rho^2}{2\alpha} < \alpha_* - \alpha^* + \rho \nu (\alpha^*  - \alpha_* \widetilde{\Lambda}).\]
 From
  \[\alpha_* - \alpha^* + \alpha^*\rho \nu  - \alpha_* \rho \nu \widetilde{\Lambda}>0,\]
  we obtain
  \[\rho > \frac{1}{\nu} \frac{\alpha^* - \alpha_*}{\alpha^* - \alpha_* \widetilde{\Lambda}}.\]
  Since $\frac{\alpha^* - \alpha_*}{\alpha^* - \alpha_* \widetilde{\Lambda}} < 1$, we get
  \[\alpha > \frac{\rho^2}{2(\alpha_*-\alpha^*+\rho\nu(\alpha^*-\alpha_*\widetilde{\Lambda}))}, \qquad
  \frac{1}{\nu} \frac{\alpha^* - \alpha_*}{\alpha^* - \alpha_* \widetilde{\Lambda}} < \rho \le \frac{1}{\nu}. \]

  \item If $\rho > \frac{1}{\nu}$, Eq.~\eqref{convergeCondition} can be written as
  \[ \frac{\rho^2}{2\alpha} < \alpha^* + \alpha_*   - \rho \nu (\alpha^*    + \alpha_* \widetilde{\Lambda}).\]
  From
  \[\alpha^* + \alpha_*   - \rho \nu (\alpha^*    + \alpha_* \widetilde{\Lambda})>0,\]
  we get
  \[\rho < \frac{1}{\nu} \frac{\alpha^* + \alpha_*}{\alpha^*    + \alpha_* \widetilde{\Lambda}}.\]
  Observing that $\frac{\alpha^* + \alpha_*}{\alpha^*    + \alpha_* \widetilde{\Lambda}}>1$, we have
  \[\alpha>\frac{\rho^2}{2(\alpha^* + \alpha_*   - \rho \nu (\alpha^*    + \alpha_* \widetilde{\Lambda}))}, \qquad
  \frac{1}{\nu} < \rho < \frac{1}{\nu} \frac{\alpha^* + \alpha_*}{\alpha^*    + \alpha_* \widetilde{\Lambda}}.\]
\end{itemize}
 The revised condition above matches the one presented in Remark 3.1 of \cite{Chen-Huang-Sheng-2017} when $\alpha_* = \alpha^* = 1$.

\section{Numerical results} \label{sec:numerical}

In this section, we report the performance of the A-H iteration for the divergence-free virtual element discretization with several examples by testing the accuracy and the robustness with respect to the parameter $\nu$. We only consider the lowest-order element ($k = 2$). Unless otherwise specified, the domain $\Omega$ is taken as the unit square $(0,1)^2$ and we set $\nu = 1$. Details of the implementation for the projections in the lowest-order case are provided in Appendix \ref{app:computation}. The influence of parameters on the A-H method can be analyzed as in \cite[Section 4.1]{Chen-Huang-Sheng-2017}.

\subsection{Test of the convergence rates}

\begin{example}\label{ex:analytic}
 We first study the incompressible N-S equations with an explicit solution given by
\[
	\bb{u}(x,y) = \begin{bmatrix}
		-\frac12\cos(x)^2\cos(y)\sin(y) \\
		\frac12\cos(y)^2\cos(x)\sin(x)
	\end{bmatrix}, \qquad
	p(x,y) = \sin x - \sin y,
\]
where the function $p$ satisfies $p\in L_0^2(\Omega)$.
\end{example}

Let $\bb{u}$ be the exact solution and $\bb{u}_h$ the discrete solution given by \eqref{NSVEM}. According to Theorem 4.6 of \cite{Beirao-Lovadina-Vacca-2018}, there holds
\begin{equation}\label{estimation}
| \bb{u}- \bb{u}_h|_1 + \| p - p_h\| \lesssim h^k,
\end{equation}
if $\bb{u}, \bb{f} \in \bb{H}^{k+1}$. Let $\bb{u}_h^n$ be the solution of the A-H iteration.
Since the VEM solution is not explicitly known inside the polygonal elements, we will evaluate the errors by comparing the exact solution $\bb{u}$ with the elliptic projection $\bb{\Pi}^K \bb{u}_h^n$. In this way, the discrete errors are quantified by
\begin{align*}
 \text{ErruH1} = \Big(\sum_{K \in \mathcal{T}_h} |\bb{u} - \bb{\Pi}^K \bb{u}_h^n|^2_{1,K}\Big)^{1/2},  \qquad  \text{ErrpL2} = \|p-p_h^n \|.
\end{align*}
The iteration stops when $\|p_h^{n+1}-p_h^n \| < h^4$, where $h$ is maximum diameter of the elements in $\mathcal{T}_h$. The parameters for the A-H method are chosen as $\rho = \frac{1}{2\nu}$ and $\alpha = \rho^2$.

To test the accuracy of the proposed method, we consider a sequence of meshes, which is a Centroidal Voronoi Tessellation of the unit square in 32, 64, 128, 256 and 512 polygons. These meshes are generated by the MATLAB toolbox~-~PolyMesher introduced in \cite{Talischi-Paulino-Pereira-2012}. We report the nodal values of the exact solution and the piecewise elliptic projection $\bb{\Pi}^K \bb{u}_h^n$ in Fig.~\ref{fig:Ex1solutions}, where $n$ denotes the final iteration number.  The discrete errors are presented in Tab.~\ref{tab:Ex1solutions}.

\begin{figure}[!thb]
	\centering
	\includegraphics[scale=0.5, trim = 80 150 80 150, clip]{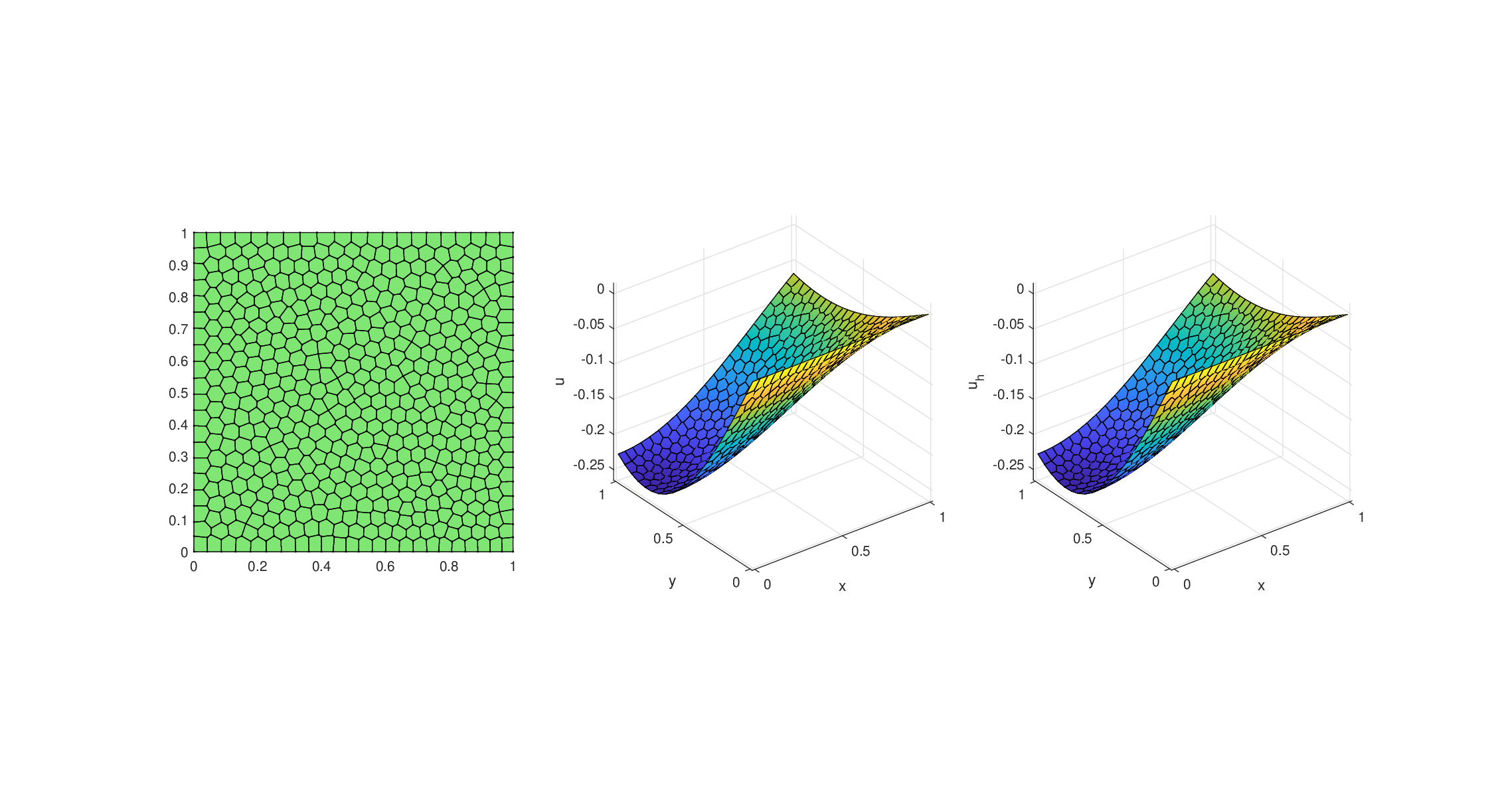}
	\caption{Numerical and exact solutions for Example \ref{ex:analytic}.}
	\label{fig:Ex1solutions}
\end{figure}

\begin{table}[!htb]
  \centering
  \caption{Discrete errors for Example \ref{ex:analytic}.} \label{tab:Ex1solutions}
  \begin{tabular}{ccccccc}
 \toprule
 $\sharp$Dof & $h$  & ErruH1 & ErrpL2   \\
  \hline
      390  & 1.768e-01    &  3.31653e-03  &  1.08419e-02\\
     774   & 1.250e-01    &  1.58320e-03  &  5.46599e-03\\
    1534   & 8.839e-02    &  7.87628e-04  &  2.80466e-03\\
    3042   & 6.250e-02    &  3.91523e-04  &  1.37330e-03\\
    6090   & 4.419e-02    &  1.96342e-04  &  6.67361e-04\\
	\bottomrule
  \end{tabular}
\end{table}

The convergence orders of the errors against the mesh size $h$ are shown in Fig.~\ref{fig:Ex1rates}. Generally speaking, $h$ is proportional to $N^{-1/2}$, where $N$ is the total number of elements in the mesh. The convergence rate with respect to $h$ is estimated by assuming $\text{Err}(h) = ch^{\alpha}$, and by computing a least squares fit to this log-linear relation.
From Fig.~\ref{fig:Ex1rates} $\left(a\right)$, we observe second-order convergence for both variables, which aligns with the theoretical prediction in  \eqref{estimation}.

\begin{figure}[!thb]
	\centering
    \subfigure[$\nu = 1$]{\includegraphics[scale=0.35]{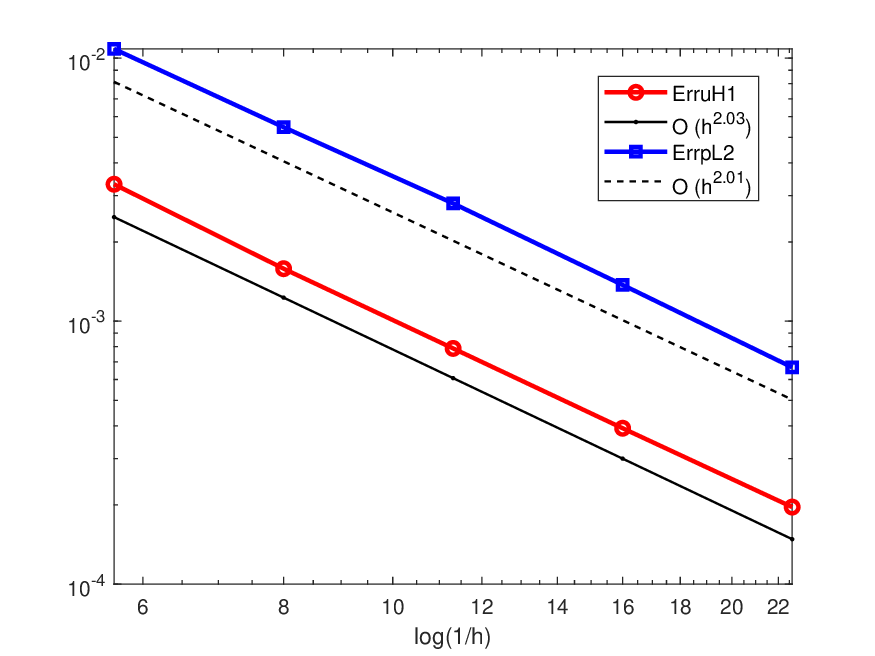}}
	\subfigure[$\nu = 0.1$]{\includegraphics[scale=0.35]{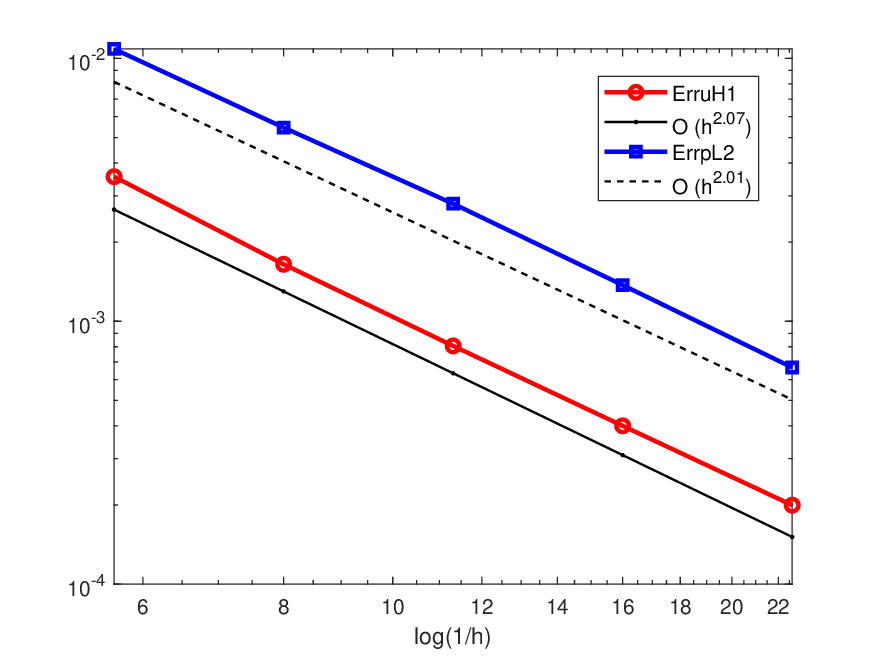}}
    \subfigure[$\nu = 0.01$]{\includegraphics[scale=0.35]{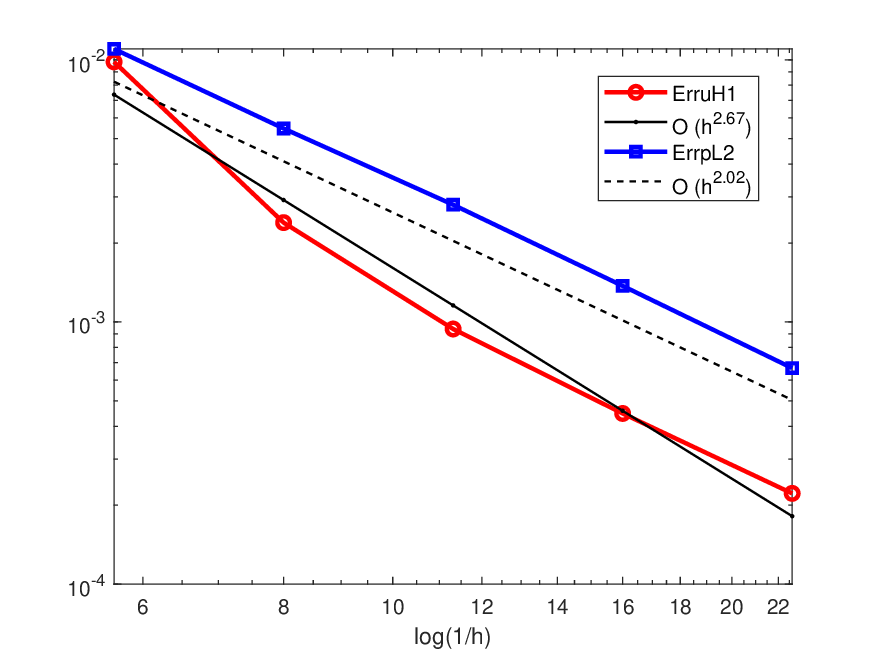}}
	\caption{Convergence rates for Example \ref{ex:analytic}.} 	\label{fig:Ex1rates}
\end{figure}

 We repeat the test for $\nu = 0.1$ and $0.01$, observing similar behavior, with the convergence orders shown in Fig.~\ref{fig:Ex1rates}(b-c) and the number of iterations provided in Tab.~\ref{tab:Ex1iteration}. Thanks to the divergence-free property of the VEM, the discrete errors remain robust with respect to $\nu$, provided that a sufficient number of iterations are performed.

\begin{table}[!htb]
  \centering
  \caption{Number of iterations for Example \ref{ex:analytic}.} \label{tab:Ex1iteration}
  \begin{tabular}{ccccccc}
			\toprule
			$\sharp$Dof & $h$  & $\nu = 1$ & $\nu=0.1$ & $\nu = 0.01$  \\
			\cmidrule(r){1-7}
			390 & 1.768e-01  & 14 & 15 & 24  \\
			774 & 1.250e-01  & 19 & 20 & 47  \\
			1534 & 8.839e-02 & 24 & 25 & 76  \\
			3042 & 6.250e-02 & 30 & 31 & 109  \\
			6090 & 4.419e-02 & 36 & 35 & 143  \\
			\bottomrule
  \end{tabular}
\end{table}

\subsection{Kovasznay flow}

\begin{example} \label{ex:Kovasznay}
Kovasznay flow is steady laminar flow on rectangle domain behind a two dimensional grid. The domain is taken as $\Omega = (-0.5, 1) \times (-0.5, 1.5)$, and the exact solution is given by
\begin{align*}
& u_1(x,y) = 1-\e^{\lambda x}\cos(2\pi y),\\
& u_2(x,y) = \frac{\lambda}{2\pi} \e^{\lambda x}\sin(2\pi y),\\
& p(x,y) = -\frac{1}{2}\e^{2\lambda x} + p_0,
\end{align*}
where $p_0$ is a constant to ensure $p\in L_0^2(\Omega)$. The parameter
\[\lambda=\frac{R_e}{2}-\sqrt{\frac{R_e^2}{4}+4\pi^2}<0.\]
\end{example}

In the numerical simulation, the partition consists of 1000 polygons and the parameters in the A-H method are chosen as before.
Figs.~\ref{fig:Kovasznay} and \ref{fig:Kovasznay1} display the streamlines of the flow with $R_e = 1$ and $40$, respectively.
Streamlines represent the paths followed by fluid particles in a flow field. These lines are tangential to the velocity vector $\bb{u}(x,y)$ at every point $(x,y)$.  In the implementation, we first compute the values of $\bb{u}$ at the nodes of $\mathcal{T}_h$, then interpolate these values onto a uniform grid. Finally, we use the Matlab built-in function streamslice.m to plot the streamlines. From the figures, we observe that the A-H based VEM accurately captures the flow dynamics.

\begin{figure}[!thb]
	\centering
    \subfigure[Exact]{\includegraphics[scale=0.5]{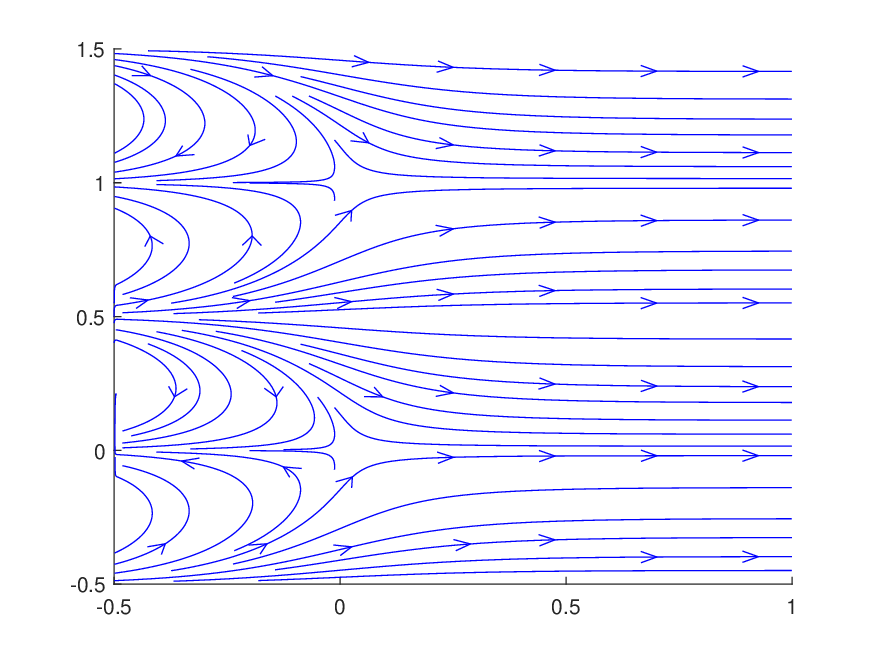}}
	\subfigure[Numerical]{\includegraphics[scale=0.5]{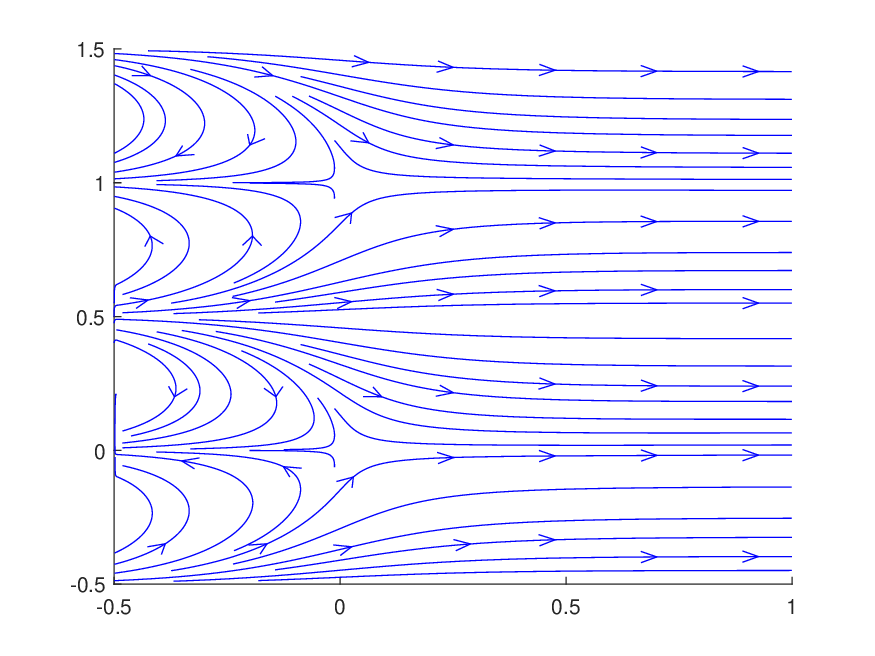}}
	\caption{Streamlines of Kovasznay flow (Example \ref{ex:Kovasznay} with $\nu = 1$).} 	\label{fig:Kovasznay}
\end{figure}

\begin{figure}[!thb]
	\centering
    \subfigure[Exact]{\includegraphics[scale=0.5]{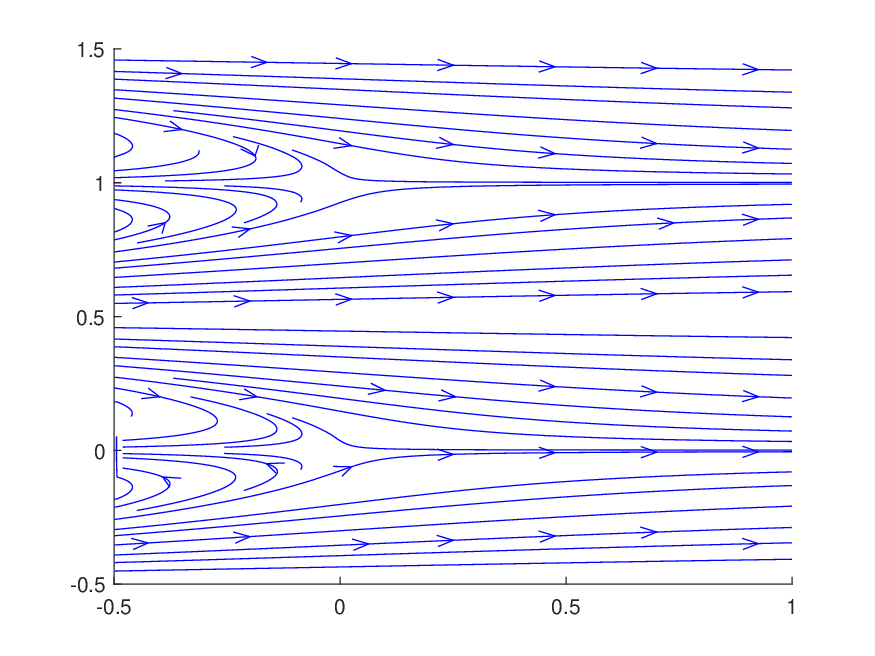}}
	\subfigure[Numerical]{\includegraphics[scale=0.5]{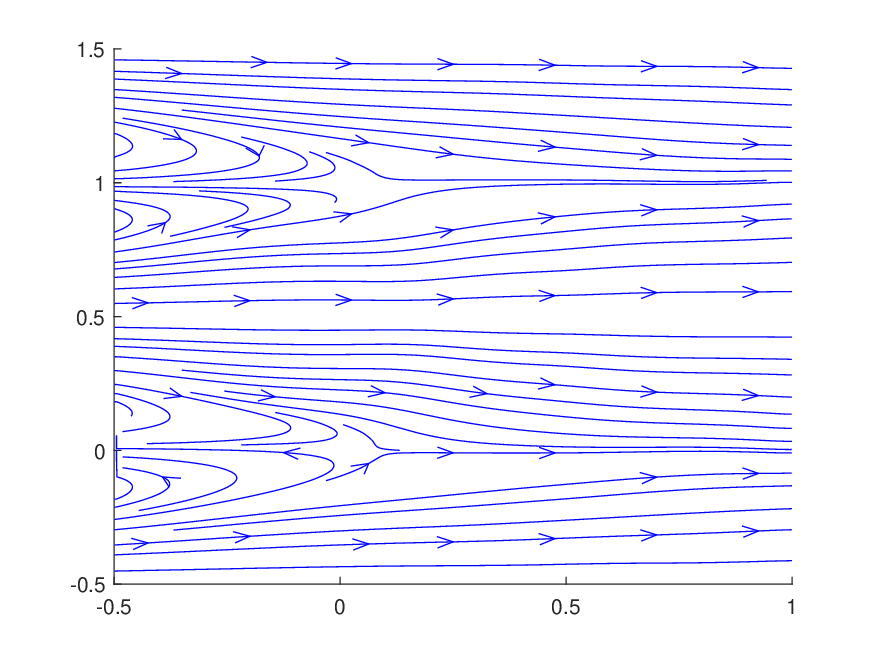}}
	\caption{Streamlines of Kovasznay flow (Example \ref{ex:Kovasznay} with $\nu = 40$).} 	\label{fig:Kovasznay1}
\end{figure}

\subsection{The lid-driven cavity problem}

\begin{example}\label{ex:cavity}
The boundary condition for the velocity on the top edge is $\bb{u} = [1, 0]^{\intercal}$, and no-slip condition is applied on the other edges. The body force is taken as $\bb{f} = \bb{0}$.
\end{example}

We use a uniform triangulation mesh in this example, and the partition consists of $64 \times 64$ cells. As for the parameters used in the A-H method, we take $\rho = 1.2$ and $\alpha = 7R_e/10$ for $R_e = 100$ as in \cite{Chen-Huang-Sheng-2017}. Fig.~\ref{fig:cavity} shows
the streamlines from the A-H method, which demonstrates close agreement with the expected physical behavior of the fluid motion as observed in Fig.~2a of \cite{Chen-Huang-Sheng-2017}.

 \begin{figure}[!thb]
	\centering
     \includegraphics[scale=0.4, trim = 80 40 0 20, clip]{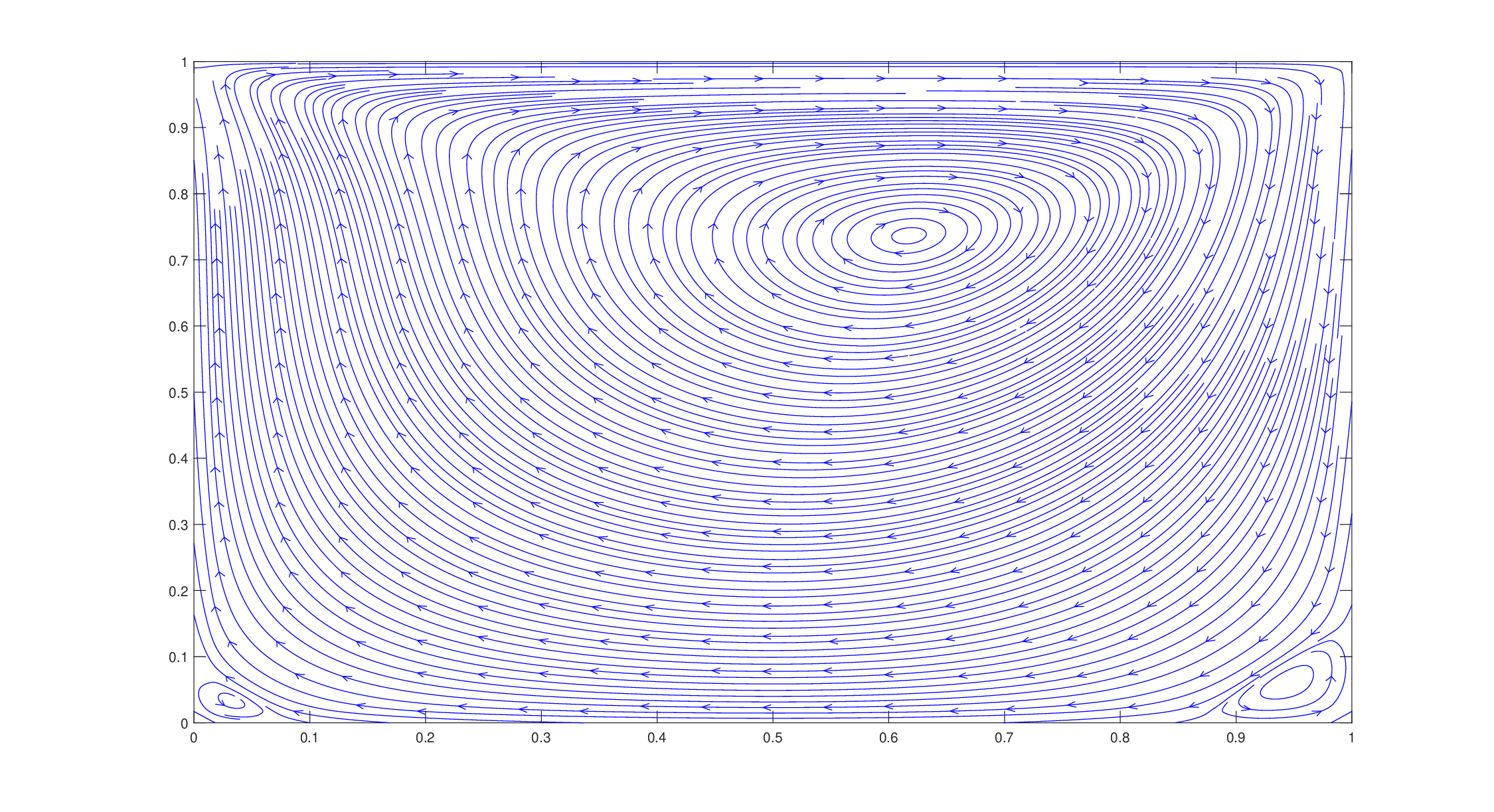} \\
	\caption{Streamlines of lid-driven cavity problem (Example \ref{ex:cavity} with $\nu = 100$).} 	\label{fig:cavity}
\end{figure}
%
%

\subsection{Extension to unsteady Navier-Stokes equations} \label{sec:unsteadyNS}

Let $\Omega \subset \mathbb{R}^2$ be a convex polygonal domain with a Lipschitz continuous boundary $\partial \Omega$. Given   an applied body force $\bb{f} \in \bb{L}^2(\Omega)$, the unsteady incompressible Navier-Stokes problem is to find the fluid velocity $\bb{u}$ and fuid pressure $\bb{p}$ such that
\begin{equation}\label{timeNS}
\begin{cases}
\bb{u}_t - \nu \Delta \bb{u} + (\bb{u} \cdot \nabla )\bb{u} - \nabla p = \bb{f} \quad & \text{in}~~\Omega \times (0,T) ,  \\
 \text{div} \bb{u} = 0 \quad  & \text{in}~~\Omega \times (0,T) ,  \\
\bb{u} = \bb{0} \quad & \text{on} ~~\partial \Omega \times (0,T),\\
\bb{u}(\cdot, t=0) = \bb{u}_0\quad & \text{on} ~~\partial \Omega \times \{0\}.
\end{cases}
\end{equation}

 For the time discretization, let us first consider a time-independent problem of the following abstract form:
\[
\bb{u}_t + L (\bb{u}, p) = \bb{f},
\]
where $L$ is a linear or nonlinear operator independent of the time variable. From time $t$ to $t + \Delta t$, the backward Euler method is given by
\[
\frac{\bb{u}^{t+\Delta t} - \bb{u}^t}{\Delta t} + L (\bb{u}^{t+\Delta t}, p) = \bb{f}^{t+\Delta t}.
\]

For simplicity, we continue to use $\bb{u}$ to denote $\bb{u}^{t+\Delta t}$ (which only depends on spatial variables) and use $\bb{u}^0$ to denote $\bb{u}^t$ (similarly for the right-hand side). Then the above equation becomes
\[
\frac{\bb{u} - \bb{u}^0}{\Delta t} + L (\bb{u}, p) = \bb{f},
\]
or equivalently,
\[
\frac{\bb{u}}{\Delta t} + L (\bb{u}, p) = \frac{\bb{u}^0}{\Delta t} + \bb{f},
\]
which reduces to a steady-state problem.

For the unsteady N-S equations, $L (\bb{u},p)$ represents the stationary part. The time-discretized variational problem can be written as
  \[
    \begin{cases}
      \frac{1}{\Delta t} (\bb{u}_h,\bb{v}_h) +  a_h(\bb{u}_h,\bb{v}_h)+N_h(\bb{u}_h;\bb{u}_h,\bb{v}_h) + b (\bb{v}_h,p_h) \\
      \hspace{1cm} =  \frac{1}{\Delta t} (\bb{u}^0,\bb{v}_h) + ( \bb{f}_h, \bb{v}_h ), \quad \bb{v}_h \in \bb{W}_h,  \\
      b(\bb{u}_h,q_h) = 0, \quad  q_h \in P_h.
    \end{cases}
  \]
  Let
  \[
    \tilde{a}_h(\bb{u}_h,\bb{v}_h) = \frac{1}{\Delta t} (\bb{u}_h,\bb{v}_h) +  a_h(\bb{u}_h,\bb{v}_h), \quad
    ( \tilde{\bb{f}}_h, \bb{v}_h ) = \frac{1}{\Delta t} (\bb{u}^0,\bb{v}_h) + ( \bb{f}_h, \bb{v}_h ).
  \]
We have
  \[
    \begin{cases}
      \tilde{a}_h(\bb{u}_h,\bb{v}_h)+N_h(\bb{u}_h;\bb{u}_h,\bb{v}_h) + b (\bb{v}_h,p_h) = ( \tilde{\bb{f}}_h, \bb{v}_h ), \quad \bb{v}_h \in \bb{W}_h,  \\
      b(\bb{u}_h,q_h) = 0, \quad  q_h \in P_h.
    \end{cases}
  \]
  This maintains the same spatial structure and can be solved using the A-H iteration.

\begin{example}\label{ex:timeNS}
 The explicit solution is given by
\[
	\bb{u}(x,y) = \begin{bmatrix}
		-\frac12 (t+1)^2 \cos(x)^2\cos(y)\sin(y) \\
		\frac12 (t+1)^2 \cos(y)^2\cos(x)\sin(x)
	\end{bmatrix}, \qquad
	p(x,y) = (t+1)(\sin x - \sin y),
\]
where the function $p$ satisfies $p\in L_0^2(\Omega)$.
\end{example}

For a given mesh $\mathcal{T}_h$, since we are using the backward Euler scheme for time discretization, it follows that $\Delta t \lesssim h^2$. Thus, we set $\Delta t = h^2$ in the implementation. The convergence rates for the first three meshes in Example \ref{ex:analytic} are presented in Fig.~\ref{fig:timeNS}, where second-order convergence for both variables are observed, consistent with the theoretical expectations.

\begin{figure}[!thb]
	\centering
    \subfigure[$\nu = 1$]{\includegraphics[scale=0.5]{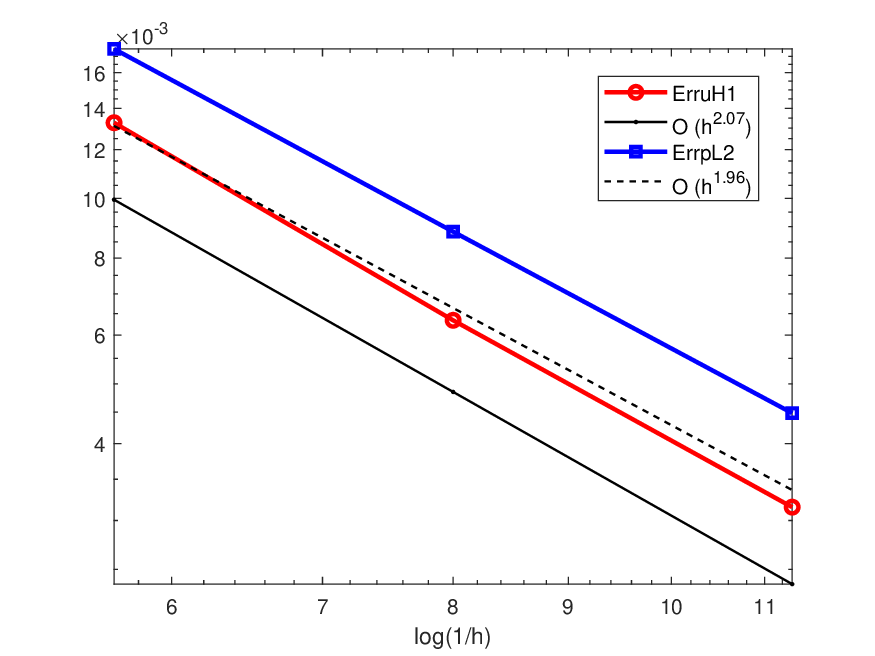}}
	\subfigure[$\nu = 0.1$]{\includegraphics[scale=0.5]{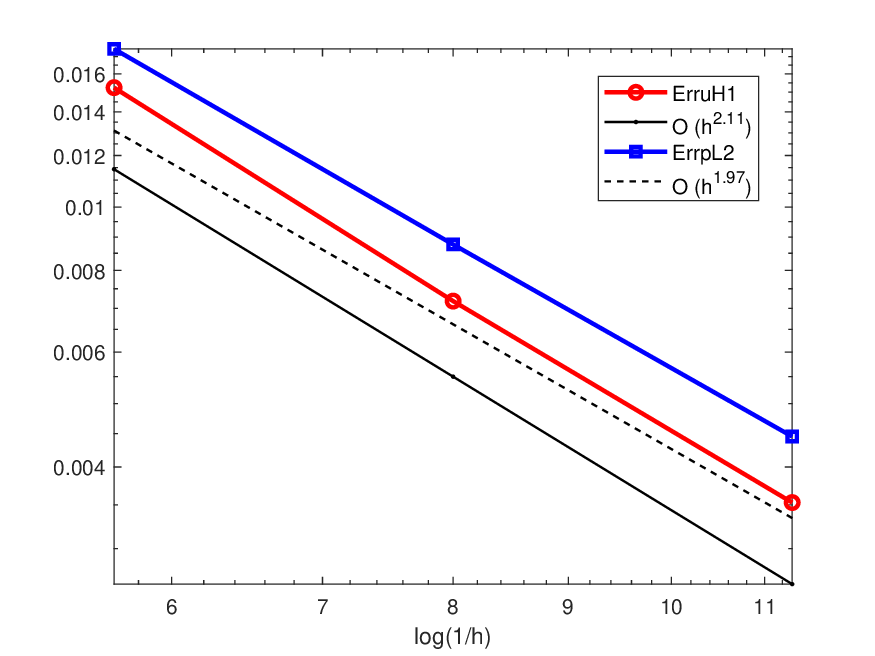}}
	\caption{Convergence rates for unsteady N-S equations (Example \ref{ex:timeNS}).} 	\label{fig:timeNS}
\end{figure}

\section{Conclusions} \label{sec:conclusion}

Divergence-free mixed VEMs are effective for discretizing steady incompressible N-S equations, resulting in nonlinear saddle-point systems. In this study, we focus on applying the A-H method from \cite{Temam1979} to solve these nonlinear saddle-point systems and analyze its convergence rate. Under several reasonable assumptions, we demonstrate that the method converges geometrically with a contraction factor independent of the mesh size.

Although the convergence rate of the A-H method does not depend on the mesh size, the CPU time increases rapidly as the virtual element mesh is refined. Therefore, optimizing and updating the method is crucial. One promising approach is combining the two-grid A-H method \cite{Du-Huang-2021-NS} with the divergence-free VEM, which will be explored in future work.

\section*{Credit authorship contribution statement}

All authors contributed equally to this work.

\section*{Declaration of competing interest}

The authors declare that they have no known competing financial interests or personal relationships that could have
appeared to influence the work reported in this paper.

\section*{Data availability}

The datasets generated during and/or analysed during the current study are available from the corresponding author on reasonable request.

\section*{Acknowledgements}

Yue Yu was partially supported by NSFC grant No. 12301561, the Key Project of Scientific Research Project of Hunan Provincial Department of Education (No. 24A0100) and the 111 Project (No. D23017). Shenxiang Cheng and Chuanjun Chen are supported by NSFC grant No. 12271468.

\appendix

\section{Implementation of the virtual element method} \label{app:computation}

We present the implementation details of the involved projections in the lowest-order case ($k=2$).

\subsection{Computation of the elliptic projection}

We first consider the elliptic projection on the original space $\bb{V}_k(K)$. The associate DoFs are arranged in the following order:
\begin{align*}
&\chi_i (\bb{v}) = \bb{v}_1(z_i),   \quad i = 1,\cdots, N_v,   \\
&\chi_{N_v+i} (\bb{v}) = \bb{v}_1(m_i),   \quad i = 1,\cdots, N_v,   \\
&\chi_{2N_v+i} (\bb{v}) = \bb{v}_2(z_i),   \quad i = 1,\cdots, N_v,   \\
&\chi_{3N_v+i} (\bb{v}) = \bb{v}_2(m_i),   \quad i = 1,\cdots, N_v,   \\
&\chi_{4N_v+1} (\bb{v}) = \frac{h_K}{|K|}\int_K \text{div} \bb{v} m_2(x,y){\mathrm d}x, \quad m_2(x,y) = \frac{x-x_K}{h_K}, \\
&\chi_{4N_v+2} (\bb{v}) = \frac{h_K}{|K|}\int_K \text{div} \bb{v} m_3(x,y){\mathrm d}x, \quad m_3(x,y) = \frac{y-y_K}{h_K}.
\end{align*}

Denote the basis functions of $\mathbb{B}_k(\partial K)$ by $\phi_1, \cdots, \phi_{N_v}; \phi_{N_v+1}, \cdots, \phi_{2N_v}$. Then the tensor-product space $(\mathbb{B}_k(\partial K))^2$ has the basis functions:
\[\overline{\phi}_1, \cdots, \overline{\phi}_{2N_v},  \underline{\phi}_{1}, \cdots, \underline{\phi}_{2N_v},\]
where
\[{\overline \phi_i} = \begin{bmatrix}
  \phi_i \\
  0
\end{bmatrix} ,\qquad
{\underline \phi_i} = \begin{bmatrix}
  0 \\
  \phi_i
\end{bmatrix},\quad i = 1, \cdots ,2N_v,\]
which correspond to the first $4N_v$ DoFs. For convenience, we denote these functions by $\bb{\phi}_1, \bb{\phi}_2, \cdots, \bb{\phi}_{4N_v}$. The basis functions associated with the last two DoFs are then denoted by $\bb{\varphi}_1 = \bb{\phi}_{4N_v+1}, \bb{\varphi}_2 = \bb{\phi}_{4N_v+2}$.

We can write the basis of $\bb{V}_k(K)$ in a compact form as
\[\bb{\phi}^{\intercal} = (\bb{\phi}_1, \bb{\phi}_2, \cdots, \bb{\phi}_{N_k}),\]
where $N_k = 4N_v+2$ is the number of the DoFs of $\bb{V}_k(K)$.  The basis of $(\mathbb{P}_k(K))^2$ is then denoted as
\[\bb{m}^{\intercal} = (\bb{m}_1, \bb{m}_2, \cdots, \bb{m}_{N_p})= (\overline{m}_1, \cdots, \overline{m}_6, \underline{m}_1, \cdots, \underline{m}_6).\]
Since $(\mathbb{P}_k(K))^2 \subset V_k(K)$, one has $\bb{m}^{\intercal} = \bb{\phi}^{\intercal}\bb{D}$, where $\bb{D}$ is referred to as the transition matrix from $(\mathbb{P}_k(K))^2$ to $V_k(K)$. Let
\[\boldsymbol{m}^{\intercal} = [{\overline m_1},{\overline m_2},{\overline m_3},{\underline m_1},{\underline m_2},{\underline m_3}] = :[\overline m^{\intercal},\underline m^{\intercal}],\]
\[{\boldsymbol \phi^{\intercal}} = [{\overline \phi_1}, \cdots ,{\overline \phi_{{N}}},{\underline \phi_1}, \cdots ,{\underline \phi_{{N}}}, \bb{\varphi}_1, \bb{\varphi}_2] = :[\overline \phi^{\intercal} ,\underline \phi^{\intercal} , \bb{\varphi}_1, \bb{\varphi}_2].\]
Let $\bb{D} = \begin{bmatrix} \bb{D}_1 \\ \bb{D}_2 \end{bmatrix}$. One has
\begin{align*}
\bb{m} = [\overline m^{\intercal},\underline m^{\intercal}]
& = [\overline \phi^{\intercal} ,\underline \phi^{\intercal} , 0, 0] \bb{D} + [\bb{0}^{\intercal} ,\bb{0}^{\intercal}  , \bb{\varphi}_1, \bb{\varphi}_2] \bb{D} \\
& = [\overline \phi^{\intercal} ,\underline \phi^{\intercal} , 0, 0] \begin{bmatrix} \bb{D}_1 \\ \bb{O} \end{bmatrix}
   + [\bb{0}^{\intercal} ,\bb{0}^{\intercal}  , \bb{\varphi}_1, \bb{\varphi}_2] \begin{bmatrix} \bb{O} \\ \bb{D}_2 \end{bmatrix}.
\end{align*}
Let $m^{\intercal} = \phi^{\intercal}D$. We can find that
\[\bb{D}_1 = \begin{bmatrix} D & \\ & D \end{bmatrix}, \qquad
\bb{D}_2 =
\begin{bmatrix}
\chi_{4N_v+1}(\overline{m}^{\intercal}) & \chi_{4N_v+1}(\underline{m}^{\intercal})\\
\chi_{4N_v+2}(\overline{m}^{\intercal}) & \chi_{4N_v+2}(\underline{m}^{\intercal})
\end{bmatrix}.\]

The elliptic projection satisfies
\begin{equation}\label{ellipticNS}
\begin{cases}
\bb{G}\bb{\Pi}_*^K = \bb{B}, \\
P_0^K(\bb{m}^{\intercal})\bb{\Pi}_*^K = P_0^K(\bb{\phi}^{\intercal})
\end{cases} \quad \mbox{or} \quad
\tilde{\bb{G}}\bb{\Pi}_*^K = \tilde{\bb{B}},
\end{equation}
where
\[{\bb{G}} = a^K({\bb{m}}, {\bb{m}}^{\intercal}), \quad {\bb{B}} = a^K({\bb{m}}, \bb{\phi}^{\intercal}).\]
When $k=2$, one has
\begin{align*}
{\bb{B}}_{\alpha i}
= a^K({\bb{m}}_\alpha, \bb{\phi}_i) = \int_K q_\alpha{\rm div}\bb{\phi}_i {\rm d}x  +  \int_{\partial K} ( \nabla {\bb{m}}_\alpha \bb{n}-q_\alpha\bb{n}) \cdot \bb{\phi}_i{\rm d}s
=:  I_1(\alpha,i) +  I_2(\alpha,i),
\end{align*}
where
\[I_1(\alpha,i) = \int_K q_\alpha{\rm div}\bb{\phi}_i {\rm d}x, \quad I_2(\alpha,i) = \int_{\partial K} ( \nabla {\bb{m}}_\alpha \bb{n}-q_\alpha\bb{n}) \cdot \bb{\phi}_i{\rm d}s,\]
and
\[\Delta {\bb{m}}_\alpha = \nabla q_\alpha + 0, \quad q_\alpha \in \mathbb{P}_1(K)/\mathbb{P}_0(K).\]

For the first term $I_1$, by definition, let $q_\alpha = h_K(c_{2,\alpha} m_2 + c_{3,\alpha} m_3)$. Then
\[\begin{bmatrix}c_{2,\alpha} \\ c_{3,\alpha} \end{bmatrix}  \longleftarrow  \begin{bmatrix} \Delta m^{\intercal} & \bb{0}^{\intercal} \\ \bb{0}^{\intercal} & \Delta m^{\intercal}\end{bmatrix}.\]
The Kronecker's property gives
\begin{align*}
I_1(\alpha,i)
 =  h_K \Big(c_{2,\alpha}\int_K  m_2{\rm div}\bb{\phi}_i {\rm d}x+ c_{3,\alpha} \int_K m_3 {\rm div}\bb{\phi}_i {\rm d}x\Big)
 =  h_K \Big(c_{2,\alpha}\delta_{i, (4N_v+1)}+ c_{3,\alpha} \delta_{i, (4N_v+2)}\Big).
\end{align*}

For the second term $I_2$, let $(g_1^\alpha,g_2^\alpha)^{\intercal} = \nabla {\bb{m}}_\alpha \bb{n}-q_\alpha\bb{n}$, Then
\begin{equation}\label{I2ai}
I_2(\alpha,i)
 = \int_{\partial K} g_1^\alpha \bb{\phi}_{1,i} {\rm d}s + \int_{\partial K} g_2^\alpha \bb{\phi}_{2,i} {\rm d}s=:J_1(\alpha,i)+J_2(\alpha,i).
\end{equation}
Let $\phi$ be the basis of $\mathbb{B}_k(\partial K)$. One has
\[J_1(\alpha,:)^{\intercal} =
\begin{bmatrix}
( g_1^\alpha, \phi )_{\partial K} \\
( g_1^\alpha, \bb{0} )_{\partial K} \\
( g_1^\alpha, \bb{\varphi}_{1,1} )_{\partial K}\\
( g_1^\alpha, \bb{\varphi}_{2,1} )_{\partial K}
\end{bmatrix}
= \begin{bmatrix}
( g_1^\alpha, \phi )_{\partial K} \\
\bb{0} \\
0\\
0
\end{bmatrix}, \qquad J_2(\alpha,:)^{\intercal} =
\begin{bmatrix}
( g_2^\alpha, \bb{0} )_{\partial K} \\
( g_2^\alpha, \phi )_{\partial K} \\
( g_2^\alpha, \bb{\varphi}_{1,2} )_{\partial K}\\
( g_2^\alpha, \bb{\varphi}_{2,2} )_{\partial K}
\end{bmatrix}
= \begin{bmatrix}
\bb{0}\\
( g_2^\alpha, \phi )_{\partial K} \\
0\\
0
\end{bmatrix},
\]
where $( g_1^\alpha, \phi )_{\partial K}$ and $( g_2^\alpha, \phi )_{\partial K}$ can be computed using the  assembly technique for FEMs. For example, for $( g_1^\alpha, \phi )_{\partial K}$ there exist three integrals on $e$:
\[F_i = \int_e g_1^\alpha \phi_i {\rm d}s, \quad i = 1,2,3,\]
where $e=[a_e,m_e,b_e]$ is an edge with $a_e$ and $b_e$ being the endpoints and $m_e$ the midpoint. By the Simpson's formula,
\[F = \begin{bmatrix} F_1 \\ F_2 \\ F_3 \end{bmatrix}
    = \frac{h_e}{6}\begin{bmatrix} g_1^\alpha(a_e) \\ g_1^\alpha(b_e) \\ 4g_1^\alpha(m_e)\end{bmatrix}.\]

It remains to consider the implementation of the constraint. At first glance the $L^2$ projection is not computable since there is no zero-order moment on $K$. In fact, the computability can be obtained using the decomposition of polynomial spaces.
Let $\bb{\phi}_i = [\bb{\phi}_{1,i},\bb{\phi}_{1,i}]^{\intercal}$. Then
\[\bb{\phi}_{1,i} = \bb{\phi}_i \cdot \begin{bmatrix} 1 \\ 0 \end{bmatrix}, \quad
\bb{\phi}_{2,i} = \bb{\phi}_i \cdot \begin{bmatrix} 0 \\ 1 \end{bmatrix}.\]
It is easy to get
\[\begin{bmatrix} 1 \\ 0 \end{bmatrix}
= \nabla p_{k-1} + \bb{g}_{k-2}^\bot, \quad
p_{k-1} = h_K  m_2 , \quad \bb{g}_{k-2}^\bot = \bb{0},\]
\[\begin{bmatrix} 0 \\ 1 \end{bmatrix}
= \nabla q_{k-1} + \bb{g}_{k-2}^\bot, \quad
q_{k-1} = h_K  m_3 , \quad \bb{g}_{k-2}^\bot = \bb{0},\]
which yield
\begin{align}
& P_0^K(\bb{\phi}_{1,i})
 = |K|^{-1} \int_K \bb{\phi}_i \cdot \nabla p_{k-1} {\rm d}x
 =  |K|^{-1}h_K\Big(- \int_K {\rm div} \bb{\phi}_i  m_2 {\rm d}x + \int_{\partial K} m_2 \bb{\phi}_i \cdot \bb{n} {\rm d}s\Big), \nonumber\\
& P_0^K(\bb{\phi}_{2,i})
 = |K|^{-1} \int_K \bb{\phi}_i \cdot \nabla q_{k-1} {\rm d}x
 = |K|^{-1}h_K\Big(- \int_K {\rm div} \bb{\phi}_i  m_3 {\rm d}x + \int_{\partial K} m_3\bb{\phi}_i \cdot \bb{n} {\rm d}s\Big). \label{constraintStokes}
\end{align}
Their computation is similar to the previous one for $B_{\alpha, i}$ with $\alpha$ fixed. The resulting two row vectors will replace the first and seventh rows of $\bb{B}$.

 For the elliptic projection on the lifting space $\widetilde{\bb{V}}_k(K)$, we need to include several additional degrees of freedom (DoFs):
\begin{equation}\label{chigU}
\chi_j^{\bot}(\bb{v}) = \frac{1}{|K|} \int_K \bb{v} \cdot \bb{g}_j^{\bot} \, \mathrm{d}x, \quad \bb{g}_j^{\bot} \in \mathcal{G}_k^\oplus(K),
\end{equation}
where $j = 1, \dots, \text{dim} (\mathcal{G}_k^\oplus(K)) = \frac{(k+1)k}{2} = 3$ (for $k=2$), and $\bb{g}_j^{\bot}$ can be chosen as
\[
\bb{g}_j^{\bot} = \bb{m}^\bot m_j, \qquad \bb{m}^\bot =  ( m_3, -m_2) m_j.
\]
By placing these DoFs after $\chi_i$, we obtain a modified transition matrix, denoted by $\bb{D}_L$, where the first $N_k = 4N_v + 2$ rows of $\bb{D}_L$ are identical to those of $\bb{D}$, and the subsequent $n_k := \frac{(k+1)k}{2} = 3$ rows are computed by means of \eqref{chigU}. The linear system in \eqref{ellipticNS} is now replaced by $\bb{G}_L\bb{\Pi}_{L*}^K = \bb{B}_L$. The right-hand side of $\bb{B}_L = a^K(\bb{m}, \bb{\phi}^{\intercal}_L)$ has $n_k = 3$ more columns than $\bb{B}$. After integration by parts, it can be seen that the computation of $a^K(\bb{m}, \bb{\rho}^{\intercal})$ does not involve the extra degrees of freedom. This means $a^K(\bb{m}, \bb{\rho}^{\intercal}) = O$, i.e., $\bb{B}_L$ has $n_k = 3$ more columns of zero vectors than $\bb{B}$. The same applies to the constraints. With the transition matrix $\bb{D}_L$ and $\bb{B}_L$, the projection matrix $\bb\Pi_L^K$ in the lifting space $\widetilde{\bb{V}}_k(K)$ can be computed.

\subsection{Computation of the $L^2$ projection}

The $L^2$ projection is defined as
\[
\begin{cases}
  \Pi_k^0:  \bb{W}_h(K) \to (\mathbb{P}_k(K))^2, \quad \bb{v} \mapsto \Pi_k^0 \bb{v},  \\
  (\Pi_k^0 \bb{v}, \bb{q})_K = (\bb{v}, \bb{q})_K.
\end{cases}
\]
Let $\Pi_k^0$ have the matrix representations $\bb{\Pi}_k^0$ and $\bb{\Pi}_{k*}^0$ with respect to the basis $\bb{\phi}^{\intercal}$ and the polynomial basis $\bb{m}^{\intercal}$, respectively,
\begin{equation}\label{L2ProjNS}
\Pi_k^0 \bb{\phi}^{\intercal} = \bb{\phi}^{\intercal}\bb{\Pi}_k^0, \qquad \Pi_k^0 \bb{\phi}^{\intercal} = \bb{m}^{\intercal}\bb{\Pi}_{k*}^0,
\end{equation}
Then,
\[
\bb{\Pi}_k^0 = \bb{D} \bb{\Pi}_{k*}^0,
\]
and
\begin{equation}\label{L2NS}
(\bb{m}, \bb{m}^{\intercal})_K \bb{\Pi}_{k*}^0 = (\bb{m}, \bb{\phi}^{\intercal})_K \quad \text{or equivalently} \quad \bb{H} \bb{\Pi}_{k*}^0 = \bb{C}.
\end{equation}
It is clear that the consistency relation $\bb{H} = \bb{C} \bb{D}$ holds.

The integral on the right-hand side needs to be computed using the polynomial decomposition
\begin{equation}\label{Pkdecomp}
(\mathbb{P}_k(K))^2 = \mathcal{G}_k(K) \oplus \mathcal{G}_k^\oplus(K) = (\nabla \mathbb{P}_{k+1}(K)) \oplus \bb{m}^\bot \mathbb{P}_{k-1}(K),
\end{equation}
where $k=2$ and $\bb{\phi}_i$ is a basis function. To this end, we first need to compute the polynomial expansion of $\text{div} \bb{\phi}_i \in \mathbb{P}_{k-1}(K)$ and $\bb{\phi}_i|_e \in (\mathbb{P}_k(e))^2$.

For $\text{div} \bb{\phi}_i \in \mathbb{P}_{k-1}(K) = \mathbb{P}_1(K)$, we can set
\[
\text{div} \bb{\phi}_i = a_{1,i} m_1 + a_{2,i} m_2 + a_{3,i} m_3.
\]
By integrating both sides with respect to $m_i$, we get the system of equations
\[
H_1 \bb{a}_i = \bb{r}_i,
\]
where $(H_1)_{i,j} = (m_j, m_i)_K$ is a fixed $3 \times 3$ matrix, and
\[
\bb{a}_i = \begin{bmatrix} a_{1,i} \\ a_{2,i} \\ a_{3,i} \end{bmatrix}, \qquad
\bb{r}_i = \begin{bmatrix}
  (m_1, \text{div} \bb{\phi}_i)_K \\
  (m_2, \text{div} \bb{\phi}_i)_K \\
  (m_3, \text{div} \bb{\phi}_i)_K
\end{bmatrix}, \qquad i = 1, \dots, N_{dof} \equiv 4N_v + 2.
\]
Let
\[
\bb{a} = [\bb{a}_1, \dots, \bb{a}_{N_{dof}}], \qquad \bb{r} = [\bb{r}_1, \dots, \bb{r}_{N_{dof}}].
\]
We have
\[
H_1 \bb{a} = \bb{r}.
\]
It is easy to see that $\bb{r} = B_K^{\intercal}$.
Based on the definition of degrees of freedom,
\[
\bb{r}(2,:) = \frac{|K|}{h_K} \cdot [\bb{0}_{4N_v}, 1, 0], \qquad
\bb{r}(3,:) = \frac{|K|}{h_K} \cdot [\bb{0}_{4N_v}, 0, 1].
\]
The first row of $\bb{r}$ can be computed using the assembly technique for finite elements. In fact, by integration by parts, we have
\begin{align*}
  \bb{r}(1,i)
  &= \int_K \text{div} \bb{\phi}_i \d x = \int_{\partial K} \bb{\phi}_i \cdot \bb{n} \, ds \\
  &= \int_{\partial K} \bb{\phi}_{1,i}  n_x \, ds + \int_{\partial K} \bb{\phi}_{2,i}  n_y \, ds =: J_1(i) + J_2(i),
\end{align*}
which can be calculated as for \eqref{I2ai}, where $\bb{n} = [n_x, n_y]^{\intercal}$. Let $\phi$ represent the basis function corresponding to $\mathbb{B}_k(\partial K)$. According to the definition of DoFs, we have
\[
J_1 =
\begin{bmatrix}
( n_x, \phi )_{\partial K} \\
( n_x, \bb{0} )_{\partial K} \\
( n_x, \bb{\varphi}_{1,1} )_{\partial K} \\
( n_x, \bb{\varphi}_{2,1} )_{\partial K}
\end{bmatrix}
= \begin{bmatrix}
( n_x, \phi )_{\partial K} \\
\bb{0} \\
0 \\
0
\end{bmatrix}, \qquad J_2 =
\begin{bmatrix}
( n_y, \bb{0} )_{\partial K} \\
( n_y, \phi )_{\partial K} \\
( n_y, \bb{\varphi}_{1,2} )_{\partial K}\\
( n_y, \bb{\varphi}_{2,2} )_{\partial K}
\end{bmatrix}
= \begin{bmatrix}
\bb{0}\\
( n_y, \phi )_{\partial K} \\
0\\
0
\end{bmatrix},
\]

For $\bb{\phi}_i|_e \in (\mathbb{P}_k(e))^2 = (\mathbb{P}_2(e))^2$, define
\[
\bb{\phi}_i|_e = (\overline{a}_{1,i}^e \overline{m}_1^e + \overline{a}_{2,i}^e \overline{m}_2^e + \overline{a}_{3,i}^e \overline{m}_{3}^e) + (\underline{a}_{1,i}^e \underline{m}_1^e + \underline{a}_{2,i}^e \underline{m}_2^e + \underline{a}_{3,i}^e \underline{m}_{3}^e),
\]
where $\overline{m}^e = [m^e, 0]^{\intercal}$, $\underline{m}^e = [0, m^e]^{\intercal}$, and
\[
m_u^e = \left( \frac{s - s_e}{h_e} \right)^{u - 1}, \quad u = 1, 2, 3,
\]
where $s$ is the arc length parameter, and $s_e$ is the midpoint of edge $e$ (in terms of the arc length parameter). The expanded form can also be written as
\[
\bb{\phi}_{1,i}|_e = \overline{a}_{1,i}^e m_1^e + \overline{a}_{2,i}^e m_2^e + \overline{a}_{3,i}^e m_3^e,
\]
\[
\bb{\phi}_{2,i}|_e = \underline{a}_{1,i}^e m_1^e + \underline{a}_{2,i}^e m_2^e + \underline{a}_{3,i}^e m_3^e,
\]
where the coefficients are determined by the values at the vertices of the edge $a_e, b_e$ and the midpoint $s_e$. For $\bb{\phi}_{1,i}|_e$, we have
\[
\bb{\phi}_{1,i}|_e(a_e) = \overline{a}_{1,i}^e + \overline{a}_{2,i}^e \left( -\frac{1}{2} \right) + \overline{a}_{3,i}^e \left( -\frac{1}{2} \right)^2,
\]
\[
\bb{\phi}_{1,i}|_e(b_e) = \overline{a}_{1,i}^e + \overline{a}_{2,i}^e \left( \frac{1}{2} \right) + \overline{a}_{3,i}^e \left( \frac{1}{2} \right)^2,
\]
\[
\bb{\phi}_{1,i}|_e(s_e) = \overline{a}_{1,i}^e,
\]
i.e.,
\[
\begin{bmatrix}
1 & -\frac{1}{2} & \frac{1}{4} \\
1 & \frac{1}{2} & \frac{1}{4} \\
1 & 0 & 0
\end{bmatrix}
\begin{bmatrix}
\overline{a}_{1,i}^e \\
\overline{a}_{2,i}^e \\
\overline{a}_{3,i}^e
\end{bmatrix} =
\begin{bmatrix}
\bb{\phi}_{1,i}|_e(a_e) \\
\bb{\phi}_{1,i}|_e(b_e) \\
\bb{\phi}_{1,i}|_e(s_e)
\end{bmatrix}, \quad i = 1, \dots, N_{dof}.
\]
According to the definition of DoFs, only the following values are non-zero:
\[
\bb{\phi}_{1,i}(z_i) = 1, \quad \bb{\phi}_{1,N_v + i}(s_i) = 1, \quad \bb{\phi}_{2,2N_v + i}(z_i) = 1, \quad \bb{\phi}_{2,3N_v + i}(s_i) = 1, \quad i = 1, \dots, N_v.
\]
For a fixed edge $e = e_j = \overline{z_j z_{j+1}}$ ($a_e = z_j, b_e = z_{j+1}, s_e = s_j$), there is the matrix equation
\[
A \overline{\bb{a}}^e = \bb{\phi}_1^e =
\begin{bmatrix}
\bb{e}_j & \bb{0}_{N_v} & \bb{0}_{N_v} & \bb{0}_{N_v} & 0 & 0 \\
\bb{e}_{j+1} & \bb{0}_{N_v} & \bb{0}_{N_v} & \bb{0}_{N_v} & 0 & 0 \\
\bb{0}_{N_v} & \bb{e}_j & \bb{0}_{N_v} & \bb{0}_{N_v} & 0 & 0
\end{bmatrix}, \quad e = e_j,
\]
where the vector $\bb{e}_j = [0, \cdots, 1, \cdots, 0]$ . Similarly, for $\bb{\phi}_{2,i}|_e$, we have
\[
A \underline{\bb{a}}^e = \bb{\phi}_2^e =
\begin{bmatrix}
\bb{0}_{N_v} & \bb{0}_{N_v} & \bb{e}_j & \bb{0}_{N_v} & 0 & 0 \\
\bb{0}_{N_v} & \bb{0}_{N_v} & \bb{e}_{j+1} & \bb{0}_{N_v} & 0 & 0 \\
\bb{0}_{N_v} & \bb{0}_{N_v} & \bb{0}_{N_v} & \bb{e}_j & 0 & 0
\end{bmatrix}, \quad e = e_j.
\]
Clearly, $\underline{\bb{a}}^e$ is simply the column-wise exchange of $\overline{\bb{a}}^e$.

Now, let's calculate the right-hand side of \eqref{L2NS}, which is given by
\[
\bb{C}(\alpha,i) = (\bb{m}_\alpha, \bb{\phi}_i)_K,  \qquad \alpha=1,\cdots,2N_m=12, ~~i = 1,\cdots, N_{dof}.
\]
For $\bb{m}_{\alpha} \in (\mathbb{P}_k(K))^2$ and $k=2$, based on the decomposition in \eqref{Pkdecomp}, we can write
\[
\bb{m}_{\alpha} = (a_{2,\alpha}^\nabla \nabla  m_2 + \cdots + a_{10,\alpha}^\nabla \nabla  m_{10})  + (a_{1,\alpha}^\bot  \bb{m}^\bot  m_1 + \cdots + a_{3,\alpha}^\bot \bb{m}^\bot  m_3).
\]
The coefficients are obtained by taking the inner product with $\nabla m_2, \cdots, \nabla m_{10}, \bb{m}^\bot  m_1, \cdots, \bb{m}^\bot  m_3$. The corresponding system of equations is
\[
H^{\oplus} \bb{a}_{\alpha}^{\oplus} = \bb{r}_{\alpha}^{\oplus}, \qquad
\bb{a}_{\alpha}^{\oplus} = [a_{2,\alpha}^\nabla,\cdots, a_{10,\alpha}^\nabla, a_{1,\alpha}^\bot,\cdots, a_{3,\alpha}^\bot]^{\intercal}.
\]

Once these coefficients are obtained, the right-hand side becomes
\begin{align*}
\bb{C}(\alpha,i) = \int_K \bb{m}_{\alpha} \cdot \bb{\phi}_i \d x
& = \sum\limits_{s=2}^{10} a_{s,\alpha}^{\nabla} \int_K \nabla m_s \cdot \bb{\phi}_i \d x
    + \sum\limits_{t=1}^3 a_{t,\alpha}^\bot \int_K  \bb{m}^\bot m_t \cdot \bb{\phi}_i \d x\\
& =: \sum\limits_{s=2}^{10} a_{s,\alpha}^{\nabla} I^\nabla(s,i)
    + \sum\limits_{t=1}^3 a_{t,\alpha}^\bot  I^\bot(t,i).
\end{align*}
Note that $I^\nabla(s,i)$ and $I^\bot(t,i)$ are independent of the index $\alpha$ and can be computed in advance. The calculations for these two terms are relatively complicated, and the details are provided below.

For the first term on the right-hand side, by integration by parts, we have
\[
I^\nabla(s,i) = \int_K \nabla m_s \cdot \bb{\phi}_i \, \d x = - \int_K \text{div}\bb{\phi}_i m_s \, \d x + \int_{\partial K} ( \bb{\phi}_i \cdot \bb{n}) m_s \, \d s , \qquad s=2,\cdots,10.
\]
Previously, we obtained
\[
\text{div} \bb{\phi}_i  =  a_{1,i} m_1 + a_{2,i} m_2 + a_{3,i} m_3 ,
\]
\[
\bb{\phi}_i|_e  =  (\overline{a}_{1,i}^e \overline{m}_1^e + \overline{a}_{2,i}^e \overline{m}_2^e  + \overline{a}_{3,i}^e \overline{m}_{3}^e)  +
(\underline{a}_{1,i}^e \underline{m}_1^e + \underline{a}_{2,i}^e \underline{m}_2^e + \underline{a}_{3,i}^e \underline{m}_{3}^e).
\]
Thus,
\begin{align*}
  I^\nabla(s,i)
  & =  - \int_K (a_{1,i} m_1 + a_{2,i} m_2 + a_{3,i} m_3) m_s \d x  \\
  & \quad + \sum\limits_{e\subset\partial K} \int_e  m_s(\overline{a}_{1,i}^e m_1^e + \overline{a}_{2,i}^e m_2^e + \overline{a}_{3,i}^e m_{3}^e)  n_x^e \d s \\
  & \quad +  \sum\limits_{e\subset\partial K} \int_e  m_s(\underline{a}_{1,i}^e m_1^e + \underline{a}_{2,i}^e m_2^e + \underline{a}_{3,i}^e m_{3}^e)  n_y^e \d s \\
  & =: -[ a_{1,i} H_1(s,1) + a_{2,i} H_1(s,2) + a_{3,i} H_1(s,3)] \\
  & \quad + \sum\limits_{e\subset\partial K} [ \overline{a}_{1,i}^e H^e(s,1) + \overline{a}_{2,i}^e H^e(s,2) + \overline{a}_{3,i}^e H^e(s,3)  ]n_x^e \\
  & \quad + \sum\limits_{e\subset\partial K} [ \underline{a}_{1,i}^e H^e(s,1) + \underline{a}_{2,i}^e H^e(s,2) + \underline{a}_{3,i}^e H^e(s,3)  ]n_y^e \\
  & = - H_1(s,:) \bb{a}_i + \sum\limits_{e\subset\partial K} H^e(s,:) (\overline{\bb{a}}_i^e n_x^e + \underline{\bb{a}}_i^e n_y^e)
  \end{align*}
or
\[
I^\nabla(s,:) =  - H_1(s,:) \bb{a} + \sum\limits_{e\subset\partial K} H^e(s,:) (\overline{\bb{a}}^e n_x^e + \underline{\bb{a}}^e n_y^e),
\]
where
\[
H_1(s, j) = \int_K m_s m_j \, \d x, \quad  s =2,\cdots 10, ~~ j =1,2,3,
\]
\[
H^e(s,j) = \int_e m_s m_j^e \, \d s, \quad s =2,\cdots 10, ~~ j =1,2,3.
\]
The matrix $H_1$ has already been obtained in the expansion of $\text{div} \bb{\phi}_i$.
For the matrix $H^e$, since $m_s m_j^e|_e \in \mathbb{P}_5(e)$, we can compute the integral by the Gauss-Lobatto formula with four points,
\[
H^e(s,j) = \int_e m_s m_j^e \, \d s = h_e \sum\limits_{p=1}^4 w_p m_s(x_e^p,y_e^p) m_j^e(s_e^p),
\]
where $(x_e^p, y_e^p)$ is computed as
\[
(x_e^p, y_e^p) = (x_e^a, y_e^a) +  s_e^p (x_e^b - x_e^a, y_e^b - y_e^a),
\]
and $s_e^p$ is the ratio of the integration points.

For the second term on the right-hand side, when $k=2$,
\[
I^\bot(t,i) = \int_K \bb{m}^\bot m_t \cdot \bb{\phi}_i \, \d x \equiv ( \bb{\phi}_i, \bb{g}_t^\bot)_K
\]
which corresponds exactly to the extra degrees of freedom in \eqref{chigU}. When $k=2$, $\mathcal{G}_{k-2}^\oplus = \{0\}$, and according to the definition of the enhancement space in \eqref{VhNS},
\[
I^\bot(t,i) = ( \bb{\phi}_i, \bb{g}_t^\bot)_K =  ( \Pi^K \bb{\phi}_i, \bb{g}_t^\bot)_K = |K| \cdot \chi_t^{\bot}(\Pi^K \bb{\phi}_i).
\]
Let the nodal basis of the lifting space $\widetilde{\bb{V}}_k(K)$ be $\bb{\phi}_L^{\intercal} = [\bb{\phi}^{\intercal}, \bb{\rho}^{\intercal}]$,
where $\bb{\phi}^{\intercal}$ corresponds to the degrees of freedom of $\bb{W}_k(K)$, and $\bb{\rho}^{\intercal}$ corresponds to the extra degrees of freedom in \eqref{chigU}.
According to the definition of degrees of freedom, the $j$-th column of $\bb{\Pi}_L^K$ is the degree of freedom vector of $\Pi_L^K \bb{\phi}_{j,L}$. Therefore, the last $n_k$ rows of $\bb{\Pi}_L^K$ are
\[
\chi_t^{\bot}( \bb{\Pi}_L^K \bb{\phi}), \qquad t = 1, \cdots, n_k = 3.
\]
By removing the last $n_k$ columns on the right-hand side, we obtain
\[
\chi_t^{\bot}( \bb{\Pi}^K \bb{\phi}), \qquad t = 1, \cdots, n_k = 3.
\]
This leads to the expression for $I^\bot(t,i)$.

\subsection{Computation of the $\Pi_{k-1}^0\nabla$ projection}

Define
\[
\Pi_{k-1}^0\nabla: \bb{W}_k(K) \to (\mathbb{P}_{k-1}(K))^{2\times 2}, \quad \bb{v}_h \mapsto \Pi_{k-1}^0\nabla \bb{v}_h,
\]
which satisfies
\[
(\Pi_{k-1}^0\nabla \bb{v}_h, \bb{P})_K = (\nabla \bb{v}_h, \bb{P})_K, \qquad \bb{P} \in (\mathbb{P}_{k-1}(K))^{2\times 2}.
\]
Since $\bb{W}_k(K)$ is not a tensor product space, $\Pi_{k-1}^0\nabla$ cannot be interpreted in component form. The basis of $(\mathbb{P}_{k-1}(K))^{2\times 2}$ can be obtained from the scaled monomials for the scalar case:
\[
\bb{P} = \begin{bmatrix} m_i & m_j \\ m_k & m_l \end{bmatrix}, \qquad m_i, m_j, m_k, m_l \in \mathbb{M}_{k-1}(K).
\]
For $k = 2$, these basis functions are arranged as
\[
\bb{M}^{\intercal} = (\bb{M}_1, \cdots, \bb{M}_{12}) = (m_1^{11}, m_2^{11}, m_3^{11}, m_1^{12}, m_2^{12}, m_3^{12}, \cdots, m_1^{22}, m_2^{22}, m_3^{22}),
\]
where $m_i^{kl}$ represents the $2\times 2$ matrix with $m_i$ at the $(k,l)$ position (and other positions are 0). The projection definition can be written as
\[
(\bb{M}, \Pi_{k-1}^0\nabla \bb{\phi}^{\intercal} )_K = (\bb{M}, \nabla \bb{\phi}^{\intercal})_K.
\]
Let $\Pi_{k-1}^0\nabla \bb{\phi}^{\intercal}$ have the matrix $\bb{\Pi}_\nabla^*$ under the basis $\bb{M}^{\intercal}$, i.e.,
\begin{equation}\label{L2nabla}
\Pi_{k-1}^0\nabla \bb{\phi}^{\intercal} = \bb{M}^{\intercal} \bb{\Pi}_\nabla^*.
\end{equation}
Then we have
\[
\bb{H}_\nabla \bb{\Pi}_\nabla^* = \bb{C}_\nabla,
\]
where
\[
\bb{H}_\nabla = (\bb{M}, \bb{M}^{\intercal})_K, \qquad \bb{C}_\nabla = (\bb{M}, \nabla \bb{\phi}^{\intercal})_K.
\]
It is easy to see that
\[
\bb{H}_\nabla = \text{diag}(H_1, H_1, H_1, H_1), \qquad H_1 = ( ( m_i, m_j)_K) _{3\times 3}.
\]

For the right-hand side $\bb{C}_\nabla(\alpha, u) = (\bb{M}_\alpha, \nabla \bb{\phi}_u)_K$, let
\[
\nabla \bb{\phi}_u = \begin{bmatrix} \partial_x \bb{\phi}_{u,1} &  \partial_y \bb{\phi}_{u,1} \\ \partial_x \bb{\phi}_{u,2} &  \partial_y \bb{\phi}_{u,2} \end{bmatrix}
= \begin{bmatrix} \partial_1 \bb{\phi}_{u,1} &  \partial_2 \bb{\phi}_{u,1} \\ \partial_1 \bb{\phi}_{u,2} &  \partial_2 \bb{\phi}_{u,2} \end{bmatrix}, \qquad u = 1, \cdots, N_{dof},
\]
with $\alpha$ corresponding to $m_i^{kl}$. Then,
\begin{align*}
\bb{C}_\nabla(\alpha, u)
& = (m_i^{kl}, \nabla \bb{\phi}_u )_K =  \int_K m_i \partial_l \bb{\phi}_{u,k} \d x \\
& = \int_{\partial K} m_i \bb{\phi}_{u,k} n_l \, ds - \int_K \partial_l m_i \bb{\phi}_{u,k} \d x
=: I_1(\alpha, u)   + I_2(\alpha, u).
\end{align*}
\begin{itemize}
  \item For the first term on the right-hand side, since $m_i \bb{\phi}_{u,k}|_e \in \mathbb{P}_3(e)$, it can be computed using finite element assembly techniques. Let $g = m_i n_l$, and let $\phi$ denote the basis function corresponding to $\mathbb{B}_k(\partial K)$, then...
      \begin{itemize}
    \item When $k=1$,
    \[I_1(\alpha,:)^{\intercal} =
\begin{bmatrix}
( g, \phi )_{\partial K} \\
( g, \bb{0} )_{\partial K} \\
( g, \bb{\varphi}_{1,1} )_{\partial K}\\
( g, \bb{\varphi}_{2,1} )_{\partial K}
\end{bmatrix}
= \begin{bmatrix}
( g, \phi )_{\partial K} \\
\bb{0} \\
0\\
0
\end{bmatrix},
\]
    \item when $k=2$,
    \[I_1(\alpha,:)^{\intercal} =
\begin{bmatrix}
( g, \bb{0} )_{\partial K} \\
( g, \phi )_{\partial K} \\
( g, \bb{\varphi}_{1,2} )_{\partial K}\\
( g, \bb{\varphi}_{2,2} )_{\partial K}
\end{bmatrix}
= \begin{bmatrix}
\bb{0}\\
( g, \phi )_{\partial K} \\
0\\
0
\end{bmatrix},
\]
  \end{itemize}

  \item For the second term on the right-hand side, since $\partial_l m_i \in \mathbb{P}_0(K)$, we have
\[
I_2(\alpha, u) = - \partial_l m_i \int_K \bb{\phi}_{u,k} \d x.
\]
This is directly obtained from the calculation in the constraint of the elliptic projection. In fact, in Eq.~$\eqref{constraintStokes}$, we have already obtained
\[
P_0^K(\bb{\phi}^\intercal) = \frac{1}{|K|} \int_K \bb{\phi}^\intercal \d x,
\]
thus
\[
I_2(\alpha, :) = - \partial_l m_i \cdot |K| \cdot \big(\text{the } k\text{-th row of } P_0^K(\bb{\phi}^\intercal)\big).
\]
\end{itemize}

\subsection{Computation of the nonlinear term}

Let $\bb{\phi}^\intercal$ be the local basis function on element $K$.
\begin{itemize}
  \item Computation of $N_h^K(\bb{\phi}_j; \bb{u}_h^n, \bb{\phi}_i)$. Let
  \[
  (N_1^n)^K_{ij} = N_h(\bb{\phi}_j; \bb{u}_h^n, \bb{\phi}_i) = \int_K \left[ (\Pi_{k-1}^0 \nabla \bb{u}_h^n) \Pi_k^0 \bb{\phi}_j \right] \cdot \Pi_k^0 \bb{\phi}_i \d x,
  \]
  where $\bb{u}_h^n = \bb{\phi}^\intercal \bb{\chi}^n$. From $\eqref{L2nabla}$ and $\eqref{L2ProjNS}$, we get
  \[
  \Pi_{k-1}^0 \nabla \bb{u}_h^n = \Pi_{k-1}^0 \nabla \bb{\phi}^\intercal \bb{\chi}^n = \bb{M}^\intercal \bb{\Pi}_\nabla^* \bb{\chi}^n =: \bb{M}^\intercal \bb{a}^n = a_1^n \bb{M}_1 + \cdots + a_{12}^n \bb{M}_{12},
  \]
  \[
  \Pi_k^0 \bb{\phi}_j = \bb{m}^\intercal \bb{\Pi}_{k*}^0(:,j) =: b_{1,j} \bb{m}_1 + \cdots + b_{12,j} \bb{m}_{12}.
  \]
  Therefore,
  \[
  (N_1^n)^K_{ij} = \sum_{r,s,t=1}^{12} a_r^n b_{sj} b_{ti} \int_K [\bb{M}_r \bb{m}_s] \cdot \bb{m}_t \d x.
  \]

  \item Computation of $N_h^K(\bb{u}_h^n; \bb{\phi}_j, \bb{\phi}_i)$. For
  \[
  (N_2^n)^K_{ij} = N_h(\bb{u}_h^n; \bb{\phi}_j, \bb{\phi}_i) = \int_K \left[ (\Pi_{k-1}^0 \nabla \bb{\phi}_j) \Pi_k^0 \bb{u}_h^n \right] \cdot \Pi_k^0 \bb{\phi}_i \d x,
  \]
  from
  \[
  \Pi_{k-1}^0 \nabla \bb{\phi}_j = \bb{M}^\intercal \bb{\Pi}_\nabla^*(:,j) =: c_{1j} \bb{M}_1 + \cdots + c_{12,j} \bb{M}_{12},
  \]
  \[
  \Pi_k^0 \bb{u}_h^n = \Pi_k^0 \bb{\phi}^\intercal \bb{\chi}^n = \bb{m}^\intercal \bb{\Pi}_{k*}^0 \bb{\chi}^n := \bb{m}^\intercal \bb{d}^n = d_1^n \bb{m}_1 + \cdots + d_{12}^n \bb{m}_{12},
  \]
  we get
  \[
  (N_2^n)^K_{ij} = \sum_{r,s,t=1}^{12} c_{rj} d_s^n b_{ti} \int_K [ \bb{M}_r \bb{m}_s] \cdot \bb{m}_t \d x.
  \]

  \item Computation of $N_h^K(\bb{u}_h^n; \bb{u}_h^n, \bb{\phi}_j)$. For
  \[
  (N^n)^K_j = N_h(\bb{u}_h^n; \bb{u}_h^n, \bb{\phi}_j) = \int_K \left[ (\Pi_{k-1}^0 \nabla \bb{u}_h^n) \Pi_k^0 \bb{u}_h^n \right] \cdot \Pi_k^0 \bb{\phi}_j \d x,
  \]
  according to the previous notation, we have
  \[
  (N^n)^K_j = \sum_{r,s,t=1}^{12} a_r^n d_s^n b_{tj} \int_K [ \bb{M}_r \bb{m}_s] \cdot \bb{m}_t \d x.
  \]
\end{itemize}
It can be seen that all three terms require the calculation of the integrals of scaled polynomials:
\[
\int_K [ \bb{M}_r \bb{m}_s] \cdot \bb{m}_t \d x, \qquad r,s,t = 1, \cdots, 12.
\]

\bibliographystyle{unsrt}
\bibliography{Refs}

\end{document}